\documentclass[a4paper,10pt,leqno]{amsart}
\title{Some closed manifolds that do not fibre over the circle}
\author{Sam Hughes}
\address{Mathematicians Institut der Universit\"at Bonn\\
                Endenicher Allee 60\\
                53115 Bonn, Germany}
              \email{sam.hughes.maths@gmail.com}\email{hughes@math.uni-bonn.de}
              \urladdr{https://samhughesmaths.github.io}
         \author{Ian Leary}
         \address{School of Mathematical Sciences\\University of Southampton \\Southampton SO17 1B\\ UK}
         \email{I.J.Leary@soton.ac.uk}
         \urladdr{https://www.southampton.ac.uk/people/5x7v5y/professor-ian-leary}
\author{Wolfgang L\"uck}
        \address{Mathematicians Institut der Universit\"at Bonn\\
                Endenicher Allee 60\\
                53115 Bonn, Germany}
         \email{wolfgang.lueck@him.uni-bonn.de}\email{lueck@math.uni-bonn.de}
         \urladdr{http://www.him.uni-bonn.de/lueck}
         \date{June, 2026}
         \makeatletter
     \@namedef{subjclassname@2020}{%
  \textup{2020} Mathematics Subject Classification}
\makeatother

 \usepackage[utf8]{inputenc}
    \usepackage{comment}
  \usepackage{calc}
  \usepackage{enumerate,amssymb,bm}
  \usepackage[arrow,curve,matrix,tips,2cell]{xy}
    \SelectTips{eu}{10} \UseTips
    \UseAllTwocells
  \usepackage{tikz}
  \usetikzlibrary{decorations.pathreplacing}
  \usepackage{color}
    \usepackage{scalerel}
  \usepackage{braket}
   \usepackage{mathtools}
  \usepackage{bbm}
  \usepackage{enumitem}
  \usepackage{float}
   \DeclareMathAlphabet{\matheurm}{U}{eur}{m}{n}
  \usepackage{hyperref}
  \usepackage{datetime2}

\DeclareMathAlphabet{\matheurm}{U}{eur}{m}{n}

\makeatletter
\def\@tocline#1#2#3#4#5#6#7{\relax
  \ifnum #1>\c@tocdepth 
  \else
    \par \addpenalty\@secpenalty\addvspace{#2}%
    \begingroup \hyphenpenalty\@M
    \@ifempty{#4}{%
      \@tempdima\csname r@tocindent\number#1\endcsname\relax
    }{%
      \@tempdima#4\relax
    }%
    \parindent\z@ \leftskip#3\relax \advance\leftskip\@tempdima\relax
    \rightskip\@pnumwidth plus4em \parfillskip-\@pnumwidth
    #5\leavevmode\hskip-\@tempdima
      \ifcase #1
       \or\or \hskip 1em \or \hskip 2em \else \hskip 3em \fi%
      #6\nobreak\relax
    \hfill\hbox to\@pnumwidth{\@tocpagenum{#7}}\par
    \nobreak
    \endgroup
  \fi}
\makeatother

\DeclareMathOperator{\cone}{cone}

\DeclareMathOperator{\F}{F}

\DeclareMathOperator{\FF}{FF}

\DeclareMathOperator{\FP}{FP}

\DeclareMathOperator{\id}{id}
\DeclareMathOperator{\im}{im}

\DeclareMathOperator{\Nil}{Nil}

\DeclareMathOperator{\pr}{pr}
\DeclareMathOperator{\res}{res}

\DeclareMathOperator{\SL}{SL}

\DeclareMathOperator{\Wh}{Wh}

  \newcommand{\IC}{\mathbb{C}}

  \newcommand{\IF}{\mathbb{F}}
  \newcommand{\IG}{\mathbb{G}}
  \newcommand{\IH}{\mathbb{H}}

  \newcommand{\IN}{\mathbb{N}}
  
  \newcommand{\IP}{\mathbb{P}}
  \newcommand{\IQ}{\mathbb{Q}}
  \newcommand{\IR}{\mathbb{R}}

  \newcommand{\IZ}{\mathbb{Z}}

  \newcommand{\cald}{\mathcal{D}}

  \newcommand{\caln}{\mathcal{N}}

  \newcommand{\calp}{\mathcal{P}}

\newcommand{\SLpar}[2]{\SL_{#1}(#2)}

\newcounter{commentcounter}

\theoremstyle{plain}
\newtheorem{theorem}{Theorem}[section]

\newtheorem{lemma}[theorem]{Lemma}

\newtheorem*{theorem*}{Theorem}
\newtheorem*{theoremA*}{Theorem A}
\newtheorem*{theoremB*}{Theorem B}
\newtheorem{addendum}[theorem]{Addendum}

\theoremstyle{definition}
\newtheorem{definition}[theorem]{Definition}
\newtheorem{example}[theorem]{Example}
\newtheorem{question}[theorem]{Question}

\newtheorem{remark}[theorem]{Remark}

\newtheorem*{definition*}{Definition}

\theoremstyle{remark}

\makeatletter\let\c@equation=\c@theorem\makeatother

\theoremstyle{definition}

\newcounter{othercommentcounter}

 \newcommand{\version}[1] 
     {\begin{center} last edited on #1\\
         last compiled on \today \ at \DTMcurrenttime.\\
         name of tex-file: \jobname
       \end{center}}

\newcommand{\RALI}{\textup{RALI}}
\newcommand{\RFRS}{\textup{RFRS}}
 \newcommand{\Sol}{\textup{Sol}}

\begin{document}

\begin{abstract}
  We construct closed manifolds  with vanishing $L^2$-Betti numbers
  (over every field) which do not virtually fibre over the circle.  The class of
  fundamental groups that occurs  is the largest possible,
  and in many cases the dimension may be taken to be six. We construct aspherical closed
  manifolds with residually (torsionfree and nilpotent) fundamental groups in all
  dimensions at least three whose $L^2$-Betti numbers vanish (over every field) and which
  do not virtually fibre over the circle.  In particular this implies that in Kielak's
  Theorem about virtually algebraic fibring for \RFRS-groups one cannot weaken the
  condition \RFRS\ to residually (torsionfree and nilpotent).
\end{abstract}

\maketitle

\newlength{\origlabelwidth} \setlength\origlabelwidth\labelwidth


\typeout{------------------- Introduction -----------------} 
\section{Introduction}\label{sec:introduction}

A major achievement of early 21st century topology is Agol's positive
resolution~\cite{Agol(2013)} to Thurston's question~\cite[Page~380]{Thurston(1982)}
whether every closed hyperbolic $3$-manifold virtually fibres over the circle $S^1$.  Here
a closed manifold $M$ \emph{fibres over the circle} if we may view it as a fibre bundle
$F\to M\to S^1$ where $F$ is a closed manifold.  A closed manifold \emph{virtually fibres
  over the circle} if it admits a finite cover $M'\to M$ such that $M'$ fibres over the
circle.

  It has long been known that finite volume even dimensional
  hyperbolic manifolds cannot virtually fibre over the circle.  The
  obstruction arises from the fact that their Euler characteristic is
  non-zero~\cite{Hirzebruch(1956)}.  But the 
  tantalising question remains whether odd dimensional closed
  hyperbolic manifolds virtually fibre over the circle, where very
  little is known in higher dimensions.  Work of
  Italiano--Martelli--Migliorini show that there exists a finite
  volume hyperbolic $5$-manifold which fibres over the
  circle~\cite{Italiano-Martelli-Migliorini(2023five)}.

Work of Farrell~\cite[Theorem~6.4]{Farrell(1971)} combined with
  developments around the Farrell--Jones Conjecture~\cite{Bartels-Lueck-Reich(2008hyper)}
  turn this into a homotopy theoretic problem (at least for dimensions at least 6).  We
  explain this in Remark~\ref{rem:Use_of_farrell_jones} and
  Remark~\ref{rem:Use_of_Farrell_Jones_and_Novikov}.  Namely, consider a (not necessarily
  aspherical) connected closed smooth manifold $M$ of dimension $\ge 6$, whose fundamental
  group $\pi$ is torsionfree and hyperbolic, finite-dimensional CAT(0), solvable, or a
  lattice, and a group epimorphism $\phi \colon \pi \to \IZ$. Then $M$ fibres over $S^1$
  in the sense that there exists a smooth fibre bundle $F \to M \xrightarrow{p} S^1$ of
  connected closed smooth manifolds with $\pi_1(p) = \phi$, if and only if the total space
  $\overline{M}$ of the infinite cyclic covering $\overline{M} \to M$ associated to $\phi$
  is homotopy equivalent to a $CW$-complex of finite type. The latter condition is
  equivalent to $\pi_1(\overline{M})$ being finitely presented and the vanishing of all
  the homology groups of the universal covering $\widetilde{M}$ with coefficients in the
  Novikov rings associated to $\phi$.  In the case the manifold $M$ is aspherical this
  amounts to showing the fundamental group of the fibre is type F.

A group $G$ is \emph{type} F$_n$ if it admits a $K(G,1)$ space with
finitely many $k$-cells for $k\leq n$.  A group is \emph{type}
F$_{\infty}$ if it is type F$_n$ for all $n$,  or, equivalently, it admits a $K(G,1)$ space with
finitely many $k$-cells for every $k \in \IZ_{\ge 0}$.  A group is \emph{type}
F if it admits a $K(G,1)$ space with finitely many cells.  Note that
F$_1$ is equivalent to finite generation and F$_2$ is equivalent to
finite presentability.  For a non-trivial ring $R$, a group is
\emph{type} FP$_n(R)$ if it admits a projective resolution $P_\ast$ of
$R$ as a trivial $RG$-module such that $P_k$ is finitely generated for
$k\leq n$.  A group is \emph{type} $\FP_{\infty}(R)$ if it is type
$\FP_n(R)$ for all $n$.  A group is \emph{type} $\FP(R)$ if it admits
a finite length projective resolution of $R$ as a trivial $RG$-module
with each term finitely generated.  Note that F$_n$ implies $\FP_n(R)$
for all $n$, that F$_1$ is equivalent to $\FP_1(R)$, and that F$_2$ is
not equivalent to $\FP_2(R)$~\cite{Bestvina-Brady(1997)}.  For one of
the finiteness properties $(P) $ just defined, we say a group $G$
\emph{virtually $(P)$-fibres} if it admits a finite index subgroup $H$
with an epimorphism $\phi:H\twoheadrightarrow \IZ$ such that
$\ker\phi$ is type $(P)$.  There exists $7$-dimensional finite volume
hyperbolic manifold $M$ such that $\pi_1(M)$ simultaneously virtually
F$_2$-fibres and virtually $\FP(\IQ)$ fibres,
see~\cite{Fisher(2023),Italiano-Martelli-Migliorini(2024up_to_eight)}.

If the universal cover of a compact manifold $M$ has a non-zero
  $L^2$-Betti number with respect to the action of the deck transformation group, then $M$
  cannot virtually fibre over the circle~\cite{Lueck(1994b)}.  
  In~\cite{Kielak(2020fibring)}, Kielak showed a partial converse to this result for the
  class of residually finite rationally solvable groups (\RFRS\ groups).  Namely, \emph{a
  finitely generated \RFRS\ group virtually F$_1$-fibres if and only if $b_1^{(2)}(G)=0$}.
  This was generalised by Fisher to higher dimensions~\cite{Fisher(2024)}: \emph{a \RFRS\ group
  of type $\FP_n(\IQ)$ virtually $\FP_n(\IQ)$-fibres if and only if $b_k^{(2)}(G)=0$ for
  $k\leq n$}.  In particular, these results apply to any odd dimensional finite volume
  hyperbolic manifold $M$ with virtually \RFRS\ fundamental group, since $b_k^{(2)}(\pi_1 M)=0$ for
  all $k$.   Many such manifolds exist after work of Agol~\cite{Agol(2008),Agol(2013)},
  Bergeron--Haglund--Wise~\cite{Bergeron-Haglund-Wise(2011)},
  Haglund--Wise~\cite{Haglund-Wise(2008),Haglund-Wise(2012)}, and Wise~\cite{Wise(2009)}.
  Analogues for finiteness properties and $L^2$-Betti numbers over other fields were also
  discussed in~\cite{Fisher(2024)}.     For a survey on the problem of fibring manifolds and groups we refer the reader
  to~\cite{Kielak(2025ICM)}.
  
  It is tempting to ask if the $L^2$-Betti numbers and their analogues in prime characteristic give
  complete obstructions to virtually fibring over the circle and if one can generalise
  these results beyond the class of \RFRS\ groups.  Note that it is indeed necessary to consider
 the prime characteristic analogues~\cite{Avramidi-Okun-Schreve(2021),Avramidi-Okun-Schreve(2024),Fisher-Hughes-Leary(2024)}.

  In this paper we present  two constructions:

  \begin{itemize}
  \item The first construction (Theorem~\ref{the:counterexamples_to_fibring_non_aspherical})
  gives manifolds of sufficiently  large dimension that have vanishing
  $L^2$-Betti numbers and that do not virtually fibre over the circle.
  The class  of fundamental groups that  occurs in Theorem~\ref{the:counterexamples_to_fibring_non_aspherical}  is the
  largest possible, and in many cases the dimension may be taken to be
  six. These manifolds are not aspherical; they even include examples in which
  the fundamental group is infinite cyclic.

\item Our second construction
  (Theorem~\ref{the:apherical_counterexamples_with_residually_nilpotent_fundamental_group})
  gives aspherical closed manifolds with residually (torsionfree and nilpotent)
  fundamental groups in all dimensions at least three whose $L^2$-Betti numbers vanish
  (over every field) and which do not virtually fibre over the circle.  In particular this
  implies that in Kielak's Theorem about virtual algebraic fibring for \RFRS-groups one
  cannot weaken the condition \RFRS\ to residually (torsionfree and nilpotent), see
  Remark~\ref{rem:RFRS_cannot_ve_replaced_by_RFRS}.  Note that such a generalisation was
  alluded to in~\cite[pages~17--18]{Kielak(2025ICM)}, asked about
  in~\cite[Question~7.2]{Fisher-Klinge(2024)}, and conjectured
  in~\cite[Conjecture~1.4]{Escartin-Ferrer(2025)}.
\end{itemize}

One might expect that direct products could be used to create higher
  dimensional examples as in our constructions using direct products.  In
  Section~\ref{sec:Products} we show that this cannot be done in any easy way.
  In an appendix we summarize some known results for 3-manifolds.


The paper is organized as follows:

\tableofcontents

\subsection*{Acknowledgments}\label{subsec:Acknowledgements}
SH was supported by a Humboldt Research Fellowship at Universit\"at Bonn.  SH and WL are
supported by the Deutsche Forschungsgemeinschaft (DFG, German Research Foundation) under
Germany's Excellence Strategy \--- GZ 2047/1, Projekt-ID 390685813, Hausdorff Center for
Mathematics at Bonn.



\typeout{------- Section 2: Preliminaries ------}
\section{Preliminaries}

\subsection{Farrell's PhD-thesis on fibring over a circle}%
\label{sec:Farrells_PhD-thesis_on_fibring_over_a_circle}

We will only need the following consequence of~\cite[Theorem~6.4]{Farrell(1971)}.

\begin{theorem}\label{the:Farrells_fibring_theorem}
  Let $M$ be a connected closed smooth manifold of dimension $\ge 6$. Consider an
  epimorphism $\phi \colon \pi_1(M) \to \IZ$. Suppose that $\Wh(\pi_1(M))$ vanishes.  Let
  $\overline{M}$ be the infinite cyclic covering associated to $\phi$.  Then the following
  assertions are equivalent:

  \begin{enumerate}
  \item\label{the:Farrells_fibring_theorem:fibring} There is a smooth locally trivial
    fibre bundle $F \to M \xrightarrow{p} S^1$ of closed smooth manifolds such that
    $\pi_1(p) \colon \pi_1(M) \to \pi_1(S^1)$ coincides with $\phi$ under the standard
    isomorphism $\IZ \xrightarrow{\cong} \pi_1(S^1)$;

  \item\label{the:Farrells_fibring_theorem_homotopy_fibre}

    $\overline{M}$ is homotopy equivalent to a finite $CW$-complex.

  \end{enumerate}
\end{theorem}
\begin{proof}\ref{the:Farrells_fibring_theorem:fibring}
  $\implies$~\ref{the:Farrells_fibring_theorem_homotopy_fibre}.  If $p$ is a locally
  trivial fibre bundle $F \to M \xrightarrow{p} S^1$ of closed smooth manifolds, then the
  preimage under $f$ of a point in $S^1$ is a closed manifold which homotopy equivalent to
  $\overline{M}$.  Any closed smooth manifold has the homotopy type of a finite
  $CW$-complex.  \\[1mm]~\ref{the:Farrells_fibring_theorem_homotopy_fibre}
  $\implies$~\ref{the:Farrells_fibring_theorem:fibring}.  Up to homotopy there is
  precisely one map $f \colon M \to S^1$ which induces $\phi$ on the fundamental groups.
  Now the implication follows directly from~\cite[Theorem~6.4]{Farrell(1971)} using the
  fact that the two obstructions $c(f)$ and $\tau(f)$ appearing there take values in
  groups which are either subgroups or quotient groups of $\Wh(\pi_1(M))$,
  see~\cite[Theorem~21]{Farrell-Hsiang(1970)}.
\end{proof}

Theorem~\ref{the:Farrells_fibring_theorem} holds also in the topological category and in
the PL-category.

\begin{remark}\label{rem:locally_trivial_bundle_versus_fibration}
  Note that in the situation of Theorem~\ref{the:Farrells_fibring_theorem} we can replace
  condition~\ref{the:Farrells_fibring_theorem:fibring} by the equivalent condition that
  there is a fibration $q \colon E \to S^1$ whose fibre is homotopic to a finite
  $CW$-complex and a homotopy equivalence $h \colon M \to E$ such that $q \circ h$ and $p$
  are homotopic.  Note that the latter condition may be phrased as fibring in the
  homotopy category. Moreover, it obviously implies
  condition~\ref{the:Farrells_fibring_theorem_homotopy_fibre} and follows from
  condition~\ref{the:Farrells_fibring_theorem:fibring}.
\end{remark}


\typeout{--------- Section 2.2: Finiteness conditions on $CW$-complexes ------}

\subsection{Finiteness conditions on \texorpdfstring{$CW$}{CW}-complexes}%
\label{sec:Finiteness_conditions_on_CW-complexes}

In order to apply Theorem~\ref{the:Farrells_fibring_theorem} one needs to handle the
problem when a given space has the homotopy type of a finite $CW$-complex which we will
discuss in this section.

 \begin{definition}[Types]\label{def:types}
   A space $X$ is \emph{of type $\F_d$} or \emph{finite $d$-type} for
   some $d \in \IZ_{\ge -1}$ if it is homotopy equivalent to a
   $CW$-complex whose $d$-skeleton is finite, and is \emph{of type
   $\F_{\infty}$} if it is homotopy equivalent to a $CW$-complex of
   finite type, i.e, a $CW$-complex all of whose skeleta are finite.
   A space $X$ is \emph{finitely-dominated} if there exists a finite
   $CW$-complex $Y$ and maps $i \colon X \to Y$ and $r \colon Y \to X$
   such that $r \circ i$ and $\id_X$ are homotopic.  Every finitely
   dominated space has the homotopy type of a $CW$-complex. We call it
   of type $\FF$ or $\FP$ respectively if it is homotopy equivalent to
   a $CW$-complex which is finite or finitely dominated respectively.
 \end{definition}

 Note that a connected space $X$ having the homotopy type of a
 $CW$-complex is always of type $\F_0$, is of type $\F_1$ if and only
 if $\pi_1(X)$ is finitely generated, and is of type $\F_2$ if and
 only if $\pi_1(X)$ is finitely presented. Moreover $X$ is of type
 $\F_{\infty}$ if and only if it is of type $\F_d$ for all $d \in
 \IZ_{\ge 0}$.  If $X$ is of type $\F_d$, $\F_{\infty}$, $\FF$, or
 $\FP$ respectively and $\widehat{X} \to X$ is a finite covering, then
 $\widehat{X}$ is of type $\F_d$, $\F_{\infty}$, $\FF$, or $\FP$
 respectively.

In the sequel all chain complexes $C_*$ are assumed to be positive, i.e., $C_n = 0$ for
$n \in \IZ_{\le -1}$.  An $S$-chain complex $C_*$ is called \emph{finitely generated
  projective} if each $C_i$ is a finitely generated projective $S$-module. We say that an
$S$-chain complex $C_*$ is \emph{finite-dimensional} if there exists $N \in \IZ_{\ge 0}$
such that $C_n \not= 0 \implies n \le N$ holds for all $n \in \IZ$.  We call an $S$-chain
complex $C_*$ \emph{finite projective} if it is both finitely generated projective and
finite-dimensional.

For $R$ a ring, $n\in\IZ_{\ge 0}$, and $C_\ast$ a chain complex of projective $R$-modules, we say
that $C_\ast$ has \emph{finite $n$-type} (\emph{over} $R$) if there is a projective
$R$-chain complex $P_*$ such that $P_i$ is finitely generated projective for every
$i \le n$ and a $R$-chain map $f_* \colon P_* \to C_*$ which is a homology equivalence,
i.e., $H_i(f_*)$ is bijective for all $i \ge 0$. 

The chain complex $C_\ast$ being finite $n$-type is equivalent to there existing a
projective $R$-chain complex $P_*$ such that $P_i$ is finitely generated projective for
$i \le n$ and $P_*$ and $C_*$ are $R$-chain homotopy equivalent. We say that  $C_\ast$
is \emph{of finite type}, if
we can choose $P_*$ in the definition above to be finitely generated projective.  This is equivalent to $C_*$ being of
finite $n$-type for every $n \in \IZ_{\ge 0}$.

\begin{lemma}\label{lem:criteria_for_finitely_dominated}
  The following assertions are equivalent for a connected $CW$-complex $X$:

  \begin{enumerate}
  \item\label{lem:criteria_for_finitely_dominated:finitely_dominated} $X$ is finitely
    dominated;

  \item\label{lem:criteria_for_finitely_dominated:finite.dimensional_and_of_finite_type}
    There is a $CW$-complex $Y$ of finite type and a finite-dimensional $CW$-complex $Z$
    such that $X$ is homotopy equivalent to both $Y$ and $Z$;

  \item\label{lem:criteria_for_finitely_dominated:chain_complex} The fundamental group
    $\pi = \pi_1(X)$ is finitely presented and the cellular $\IZ \pi$-chain complex
    $C_*(\widetilde{X})$ of its universal covering $\widetilde{X}$ is $\IZ \pi$-chain
    homotopy equivalent to a finite projective $\IZ \pi$-chain complex.

  \end{enumerate}
\end{lemma}
\begin{proof}
  See~\cite{Wall(1966)} or in the more general equivariant setting~\cite[Proposition~11.11
  on page~222 and Proposition~14.9 on page~282]{Lueck(1989)}.
\end{proof}

\begin{remark}\label{rem:Walls_finiteness_obstruction}
  Let $X$ be a finitely dominated connected $CW$-complex with fundamental group
  $\pi = \pi_1(X)$.  Then there is an obstruction $\widetilde{o}(X)$ in the reduced
  projective class group $\widetilde{K}_0(\IZ \pi)$ which vanishes if and only if $X$ is
  homotopy equivalent to a finite $CW$-complex, see Wall~\cite{Wall(1965a),Wall(1966)}, or
  in the more general equivariant setting~\cite[Theorem~14.6 on page 278]{Lueck(1989)}.
  
  Hence $X$ is automatically homotopy equivalent to a finite $CW$-complex if
  $\widetilde{K}_0(\IZ \pi)$ vanishes.
\end{remark}

\begin{theorem}\label{the:Farrells_fibring_theorem_improved}
  Let $M$ be a connected closed smooth manifold of dimension $\ge 6$. Consider an
  epimorphism $\phi \colon \pi_1(M) \to \IZ$. Suppose that both $\Wh(\pi_1(M))$ and
  $\widetilde{K}_0(\IZ[\ker(\phi)])$ vanish.  Let $\overline{M}$ be the infinite cyclic
  covering associated to $\phi$.  Then the following assertions are equivalent:

  \begin{enumerate}
  \item\label{the:Farrells_fibring_theorem_improved:fibring} There is a smooth locally
    trivial fibre bundle $F \to M \xrightarrow{p} S^1$ of closed smooth manifolds such
    that $\pi_1(p) \colon \pi_1(E) \to \pi_1(S^1)$ coincides with $\phi$ under the
    standard isomorphism $\IZ \xrightarrow{\cong} \pi_1(S^1)$;

  \item\label{the:Farrells_fibring_theorem_improved:homotopy_fibre_finite}

    $\overline{M}$ is homotopy equivalent to a finite $CW$-complex;

  \item\label{the:Farrells_fibring_theorem_improved:homotopy_fibre_of_finite_type}

    $\overline{M}$ is homotopy equivalent to $CW$-complex of finite type;

  \item\label{the:Farrells_fibring_theorem_improved:fin_proj_chain_complex}

    The fundamental group $\pi_1(\overline{M})$ is finitely presented and the
    $\IZ[\pi_1(\overline{M})]$-chain complex $i^* C_*(\widetilde{M})$ obtained from the
    cellular $\IZ[\pi_1(M)]$-chain complex $C_*(\widetilde{M})$ by restriction with the
    inclusion $i \colon \pi_1(\overline{M}) = \ker(\phi) \to \pi_1(M)$ is
    $\IZ[\pi_1(\overline{M})]$-chain homotopy equivalent to finite projective
    $\IZ[\pi_1(\overline{M})]$-chain complex;

  \item\label{the:Farrells_fibring_theorem_improved:fin_gen_pro_jchain_complex}

    The fundamental group $\pi_1(\overline{M})$ is finitely presented and the
    $\IZ[\pi_1(\overline{M})]$-chain complex $i^* C_*(\widetilde{M})$ is
    $\IZ[\pi_1(\overline{M})]$-chain homotopy equivalent to (not necessarily
    finite-dimensional) finitely generated projective $\IZ[\pi_1(\overline{M})]$-chain
    complex.

  \end{enumerate}
\end{theorem}
\begin{proof}~\ref{the:Farrells_fibring_theorem_improved:fibring}
  $\Longleftrightarrow$~\ref{the:Farrells_fibring_theorem_improved:homotopy_fibre_finite}
  This follows from Theorem~\ref{the:Farrells_fibring_theorem}.
  \\[1mm]~\ref{the:Farrells_fibring_theorem_improved:homotopy_fibre_finite}
  $\Longleftrightarrow$\ref{the:Farrells_fibring_theorem_improved:homotopy_fibre_of_finite_type}
  This follows from Lemma~\ref{lem:criteria_for_finitely_dominated} and
  Remark~\ref{rem:Walls_finiteness_obstruction}, since $\overline{M}$ is a smooth manifold
  and hence homotopy equivalent to a $CW$-complex of the finite dimension
  $\dim(\overline{M})$.
  \\[1mm]~\ref{the:Farrells_fibring_theorem_improved:homotopy_fibre_finite}
  $\implies$~\ref{the:Farrells_fibring_theorem_improved:fin_proj_chain_complex} This
  follows from Lemma~\ref{lem:criteria_for_finitely_dominated}.
  \\[1mm]~\ref{the:Farrells_fibring_theorem_improved:fin_proj_chain_complex}
  $\implies$~\ref{the:Farrells_fibring_theorem_improved:homotopy_fibre_of_finite_type}
  This also follows from Lemma~\ref{lem:criteria_for_finitely_dominated}.
  \\[1mm]~\ref{the:Farrells_fibring_theorem_improved:fin_proj_chain_complex}
  $\implies$~\ref{the:Farrells_fibring_theorem_improved:fin_gen_pro_jchain_complex} This
  is is obvious.
  \\[1mm]~\ref{the:Farrells_fibring_theorem_improved:fin_gen_pro_jchain_complex}
  $\implies$~\ref{the:Farrells_fibring_theorem_improved:fin_proj_chain_complex} This
  follows from~\cite[Proposition~11.10 on page~221]{Lueck(1989)} since
  $i^* C_*(\widetilde{M})$ is $\IZ[\pi_1(\overline{M})]$-chain homotopy equivalent to a
  finite-dimensional $\IZ[\pi_1(\overline{M})]$-chain complex.
\end{proof}


\typeout{--------- Section 2.3: Consequences of the Farrell--Jones Conjecture ------}

\subsection{Consequences of the Farrell--Jones Conjecture}%
\label{sec:Consequences_of_the_Farrell-Jones_Conjecture}

A group $G$ is called a \emph{Farrell--Jones group} if it satisfies the
so-called \emph{Full Farrell--Jones Conjecture} as formulated, for
example, in~\cite[Conjecture~13.30 on page~387]{Lueck(2022book)}. The
full statement is quite complicated and involves $L$-theory as well as
$K$-theory and the precise formulation is not relevant for this paper.
But the following facts are important for us, see~\cite[Theorem~16.1
  on page~481 and Theorem~13.65 on page~405]{Lueck(2022book)}.

\begin{itemize}
\item The class of Farrell--Jones groups contains hyperbolic groups, finite dimen\-sio\-nal
  CAT(0)-groups, virtually solvable groups, lattices in path connected second countable
  locally compact Hausdorff groups, fundamental groups of (not necessarily compact)
  connected manifolds (possibly with boundary) of dimension $\le 3$, and $S$-arithmetic
  groups;

\item The class of Farrell--Jones groups has the following inheritance properties: it is
  closed under the passage to subgroups, to overgroups of finite index, to finite direct
  products, finite free products, and to colimits over directed systems (with arbitrary
  structure maps);

\item If $G$ is a Farrell--Jones group and is torsionfree, then $\Wh(G)$ and
  $\widetilde{K}_0(\IZ[G])$ vanish.

\end{itemize}

\begin{remark}\label{rem:Use_of_farrell_jones}
  For us the following consequence is interesting. Let $M$ be a connected closed smooth
  manifold of dimension $\ge 6$ and $\phi \colon \pi_1(M) \to \IZ$ be an epimorphism of
  groups. Assume that $\pi_1(M)$ is a torsionfree Farrell--Jones group.  Then both
  $\Wh(\pi_1(M))$ and $\widetilde{K}_0(\IZ[\ker(\phi)])$ vanish. Hence the following
  assertions are equivalent by Theorem~\ref{the:Farrells_fibring_theorem_improved}:
  \begin{itemize}
  \item There is a smooth locally trivial fibre bundle $F \to M \xrightarrow{p} S^1$ of
    closed smooth manifolds such that $\pi_1(p) \colon \pi_1(E) \to \pi_1(S^1)$ coincides
    with $\phi$ under the standard isomorphism $\IZ \xrightarrow{\cong} \pi_1(S^1)$;

  \item The fundamental group $\pi_1(\overline{M})$ is finitely presented and the
    $\IZ[\pi_1(\overline{M})]$-chain complex $i^*C_*(\widetilde{M})$ is
    $\IZ[\pi_1(\overline{M})]$-chain homotopy equivalent to a (not necessarily
    finite-dimensional) finitely generated projective $\IZ[\pi_1(\overline{M})]$-chain
    complex.
  \end{itemize}
\end{remark}


\typeout{--------- Section 2.4: The Novikov ring ------}

\subsection{The Novikov ring}%
\label{sec:The_Novikov_ring}

In view of Theorem~\ref{the:Farrells_fibring_theorem_improved} the following problem
occurs.  Let $\phi \colon G \to \IZ$ be a surjective group homomorphism and $C_*$ be a
finite projective $\IZ[G]$-chain complex.  Let $K$ be the kernel of $\phi$ and
$i \colon K \to G$ be the inclusion.  Let $i^*C_*$ be the projective $\IZ[K]$-chain
complex obtained from $C_*$ by restriction via~$i$. We want to decide whether $i^*C_*$ is
$\IZ[K]$-chain homotopy equivalent to a finite projective $\IZ[K]$-chain complex.  This
can be done in terms of the Novikov ring as explained next.

Let $S$ be a ring. Let $\Psi \colon S \xrightarrow{\cong} S$ be a ring automorphism.  Let
$S_{\Psi}[t,t^{-1}]$ be the \emph{ring of $\Psi$-twisted finite Laurent series}, i.e.,
formal sums $\sum_{n \in \IZ} s_nt^n$ for $s_n \in S$, where only finitely many of the
coefficients $s_n$ are different from zero and the multiplication is given by the formula
$(s_1t^{n_1}) \cdot (s_2t^{n_2}) = (s_1\Psi^{n_1}(s_2))t^{n_1 + n_2}$. The \emph{Novikov
  rings} $S_{\Psi}((t))$ and $S_{\Psi}((t^{-1}))$ are the completions of
$S_{\Psi}[t,t^{-1}]$ defined by
\begin{eqnarray*}
  S_{\Psi}((t)) & = & \left\{\sum_{n \in \IZ} s_n t^n \mid \exists N \in \IZ_{\le 0}
                      \;\text{satisfying}\; s_n \not = 0 \implies n \ge N\right\};
  \\
  S_{\Psi}((t^{-1})) & = &  \left\{\sum_{n \in \IZ} s_n t^n \mid \exists N \in \IZ_{\ge 0}
                           \text{satisfying}\; s_n \not = 0 \implies n \le N\right\}.
\end{eqnarray*}
The multiplication is given by the obvious formula and we have
\[
  S_{\Psi}((t)) \cap S_{\Psi}((t^{-1})) = S_{\Psi}[t,t^{-1}].
\]

\begin{lemma}\label{lem:Novikov_rings_and_chain_complex}
  Let $C_*$ be a finite projective $S_{\Psi}[t,t^{-1}]$-chain complex. Denote by $i^*C_*$
  the $S$-chain complex obtained from $C_*$ by restriction with the inclusion
  $i \colon S \to S_{\Psi}[t,t^{-1}]$. Then the following assertions are equivalent:

  \begin{enumerate}
  \item\label{lem:Novikov_rings_and_chain_complex:fin-dom} The $S$-chain complex $i^*C_*$
    is $S$-chain homotopy equivalent to a finite projective $S$-chain complex;

  \item\label{lem:Novikov_rings_and_chain_complex:Novikov} The homology modules
    \[
      H_n(S_{\Psi}((t)) \otimes_{S_{\Psi}[t,t^{-1}]} C_*) \text{ and }   H_n(S_{\Psi}((t^{-1})) \otimes_{S_{\Psi}[t,t^{-1}]} C_*)
    \]
    vanish for every
    $n \in \IZ_{\ge 0}$.

  \end{enumerate}
\end{lemma}
\begin{proof} The proof for trivial $\Psi$ can be found
  in~\cite[Theorem~2]{Ranicki(1995Novikovring)} which extends to the twisted case.
\end{proof}

 \begin{example}\label{exa:Novikov_rings_and_group_rings}
   Here is the main example for $S$ and $\Psi$. Consider a group epimorphism
   $\phi \colon G \to \IZ$ with kernel $K$.  Let $y \in G$ be a fixed element which is
   mapped to a generator of $\IZ$ under $\phi$. Let
   $\gamma \colon K \xrightarrow{\cong} K$ be the group automorphism of $K$ given by
   conjugation with $y$.  Then there is an obvious group isomorphism
   \[
     K \rtimes_{\gamma} \IZ \xrightarrow{\cong} G
   \]
   sending $k \in K$ to $k$ and a generator of $\IZ$ to $y$. It induces for every ring $R$
   a ring isomorphism
   \[
     R[K]_{R[\gamma]}[t,t^{-1}] \xrightarrow{\cong} RG.
   \]
 \end{example}

\begin{remark}\label{rem:Use_of_Farrell_Jones_and_Novikov}
  For us the following consequence is relevant. Let $M$ be a connected closed smooth
  manifold of dimension $\ge 6$.  Consider an epimorphism $\phi \colon \pi_1(M) \to \IZ$.
  Let $\overline{M} \to M$ be the infinite cyclic covering associated to $\phi$.  Assume
  that $\pi_1(M)$ is a torsionfree Farrell--Jones group.  Then the following assertions are
  equivalent using the notation of Example~\ref{exa:Novikov_rings_and_group_rings}:
  \begin{itemize}
  \item There is a smooth locally trivial fibre bundle $F \to M \xrightarrow{p} S^1$ of
    closed smooth manifolds such that $\pi_1(p) \colon \pi_1(E) \to \pi_1(S^1)$ coincides
    with $\phi$ under the standard isomorphism $\IZ \xrightarrow{\cong} \pi_1(S^1)$;

  \item The fundamental group $\pi_1(\overline{M})$ is finitely presented and for
    $n \in \IZ_{\ge 0}$ both
    $H_n(\IZ[K]_{\IZ[\gamma]}((t)) \otimes_{\IZ[K]}i^*C_*(\widetilde{M}))$ and
    $H_n(\IZ[K]_{\IZ[\gamma]}((t^{-1})) \otimes_{\IZ[K]}i^*C_*(\widetilde{M}))$ vanish.
  \end{itemize}
\end{remark}


\typeout{--------- Section 2.5: $L^2$-Betti numbers are obstruction to fibring over
  S^1------}

\subsection{\texorpdfstring{$L^2$}{L2}-Betti numbers are an obstruction to fibring over \texorpdfstring{$S^1$}{S1}}%
\label{sec:L2-Betti_numbers_are_obstruction_to_fibring_over_S1}

$L^2$-Betti numbers are obstructions to fibring over $S^1$ as explained next.  Let
$F \to E \to S^1$ be a fibration such that $F$ has the homotopy type of a $CW$-complex of
finite type.  Then $E$ has the homotopy type of a $CW$-complex of finite type and is
actually homotopy equivalent to a mapping torus $T_f$ of a selfmap $f \colon Y \to Y$ of
some $CW$-complex $Y$ which is of finite type and homotopy equivalent to $F$.  This does
imply that the $L^2$-Betti numbers $b_n^{(2)}(\widetilde{X})$ vanish for $n \in \IZ$,
see~\cite[Theorem~6.63 on page~270]{Lueck(2002)}.

An analogous statement holds also for the $\IF_p$-version of the $L^2$-Betti numbers
$b_n(\widetilde{E};\cald_{\IF_p[\pi_1(E)]})$,
see~\cite[Theorem~3.25]{Avramidi-Lueck(2026)}. These numbers are defined based on the work
of Jaikin-Zapirain and systematically studied by Avramidi and
L\"uck~\cite{Avramidi-Lueck(2026)}, provided that $\pi_1(E)$ is a \RALI-group, i.e., is
residually (amenable and locally indicable).

Note that the class of finitely generated RALI-group is larger than the class of finitely
generated \RFRS-groups, since a finitely generated group is \RFRS\ if and ony if it
residually (locally indicable and virtually abelian),
see~\cite[Section~6]{Okun-Schreve(2024orders)}.


\typeout{--------- Section 2.6: \texorpdfstring{$L^2$}{L2}-Betti numbers and finiteness of the homotopy fibre  ------}

\subsection{\texorpdfstring{$L^2$}{L2}-Betti numbers and algebraic
  fibring}\label{sec:L2-Betti_numbers_and_algebraic_fibring}

A group $G$ is \emph{algebraically fibred} if it admits a homomorphism
$\phi \colon G \to \IZ$ whose kernel is finitely generated.  Next we record the celebrated
result of Kielak~\cite[Theorem~5.3]{Kielak(2020fibring)}

\begin{theorem}\label{the:Kielak}
  Let $G$ be an infinite finitely generated virtually $\RFRS$-group.  Then $G$ virtually
  algebraically fibres if and only if its first $L^2$-Betti number
  $b_1^{(2)}(G) = b_1^{(2)}(EG;\caln(G))$ vanishes.
\end{theorem}

This was extended by Fisher~\cite[Theorem~A]{Fisher(2024)} to higher dimensions.

\begin{theorem}\label{the:Fisher_Q}
  Let $n \in \IZ_{\ge 1}$.  Let $G$ be a virtually $\RFRS$-group of type
  $\FP_n(\IQ)$. Then the following are equivalent:

  \begin{enumerate}
  \item\label{the:Fisher:(1)} We have $b_p^{(2)}(G) = $ for $p \le n$;
  \item\label{the:Fisher_Q:(2)} There is a finite-index subgroup $H\subseteq G$ and an
    epimorphism $\phi \colon H \to \IZ$ such that $\ker(\phi)$ is of type $\FP_n(\IQ)$;

  \item\label{the:Fisher_Q:(3)} There is a finite-index subgroup $H' \subseteq G$ and an
    epimorphism $\phi' \colon H' \to \IZ$ such that $b_p(H') < \infty $ holds for
    $p \le n$.
  \end{enumerate}
\end{theorem}

There is also the following version for an arbitrary field $F$ due to
Fisher~\cite[Theorem~B]{Fisher(2024)}.

 \begin{theorem}\label{the:Fisher_F}
   Let $F$ be a field and let $n \in \IZ_{\ge 1}$.  Let $G$ be a virtually $\RFRS$-group
   of type $\FP_n(F)$. The following are equivalent:

   \begin{enumerate}
   \item\label{the:Fisher_F:(1)} We have $b_p(G;\cald_{F[G]}) = 0$ for $p \le n$;
   \item\label{the:Fisher_F:(2)} There is a finite-index subgroup $H\subseteq G$ and an
     epimorphism $\phi \colon H \to \IZ$ such that $\ker(\phi)$ is of type $\FP_n(F)$;

   \item\label{the:Fisher_F:(3)} There is a finite-index subgroup $H' \subseteq G$ and an
     epimorphism $\phi' \colon H' \to \IZ$ such that $b_p(H';F)< \infty$ holds for
     $p \le n$.
   \end{enumerate}
 \end{theorem}

It was shown by Fisher--Italiano--Kielak that one can also combine the finiteness properties of
the kernel~\cite[Theorem~5.3]{Fisher-Italiano-Kielak(2025)}.  Here is a variation of their
result.

Consider any finite free $\IZ[G]$-chain complex $C_*$, any prime $p$, and any field $F$ of
characteristic $p$.  Let $\IF_p$ be the finite field consisting of $p$-elements.  Then
\begin{eqnarray*}
  b_n^{(2)}(C_*;\cald_{F[G]})  & = & b_n^{(2)}(C_*;\cald_{\IF_p[G]});
  \\
  b_n^{(2)}(C_*;\caln(G)) & \le &b_n^{(2)}(C_*;\cald_{\IF_p[G]}),
\end{eqnarray*}
and there is a finite set of primes $\calp_{C_*}$ (depending on $C_*$) such that for every
$p \not \in \calp_{C_*}$ we get
$b_n^{(2)}(C_*;\caln(G)) = b_n^{(2)}(C_*;\cald_{\IF_p[G]})$, see~\cite[Lemma~3.9 and
Theorem~4.26]{Avramidi-Lueck(2026)}.  If $\IF_p[G] \otimes_{\IZ[G]} C_*$ is of finite
$n$-type over $\IF_p$ then $F[G] \otimes_{\IZ[G]} C_*$ is of finite $n$-type over $F$.
Hence the chain complex version of Theorem~\ref{the:Fisher_F} implies the next result.

\begin{theorem}\label{the:FIK_new}
  Let $G$ be a virtually $\RFRS$ group of type $\FP_{n+1}$.  Then there exists a finite
  set of primes $\calp_G$ with the following properties:

  \begin{enumerate}
  \item\label{the:FIK_new:calp_G} The following statements are equivalent:
    \begin{enumerate}
    \item\label{the:FIK_new:calp_G:vanishing}We have $b_k(G;\cald_{\IF_p[G]})=0$ for
      $k\le n$ and $p \in \calp_G$;
    \item\label{the:FIK_new:calp_G:type}There exists a finite-index subgroup
      $H\subseteq G$ and an epimorphism $\phi\colon H \to \IZ$ such that $\ker\phi$ is
      type $\FP_n(F)$ for every field $F$;
    \end{enumerate}
  
  \item\label{the:FIK_new:calp} Let $\calp$ be a any non-empty set of primes. Then the
    following statements are equivalent:
    \begin{enumerate}
    \item\label{the:FIK_new:calp:vanishing} We have $b_k(G;\cald_{\IF_p[G]})=0$ for
      $k\le n$ and $p \in \calp$;
    \item\label{the:FIK_new:calp:type} There exists a finite-index subgroup $H\subseteq G$
      and an epimorphism $\phi\colon H \to \IZ$ such that $\ker\phi$ is type $\FP_n(F)$
      for every field $F$ of characteristic $p$ satisfying $p \in \calp$  or $p \notin \calp_G$.
    \end{enumerate}
  \end{enumerate}
\end{theorem}


\typeout{-------------------- Section 2.G: BNSR Invariants --------------------}

\subsection{BNSR invariants of groups and spaces}

Let $G$ be a finitely generated group and define
$S(G)=\hom(G;\IR)\smallsetminus\{0\}$. Note that we are always considering $\hom(G,\IR)$
with the usual topology (in fact, since our group $G$ is finitely generated, this is
homeomorphic to $\IR^n$ for some $n$). Given $\phi\in S(G)$, define a submonoid of $G$ by
\[{G}_{\phi}\coloneqq\left\{ g\in G : \phi(g)\geq 0\right\}.\]

Let $C_\ast$ be a chain complex of $RG$-modules and let $k\in\IN$.  Note that we may view
$C_\ast$ as a chain complex of $RG_\varphi$-modules by restriction.  We define
\[\Sigma^k(C_\ast;R)\coloneqq\{\varphi\in S(G) : C_\ast \text{ is
    of finite $k$-type over }RG_\varphi\}.\]

Let $X$ be a connected CW complex with $C_\ast(X)$ of finite $n$-type over $R$.  For
$k\leq n$ the \emph{$k$-th (homological) BNSR invariant of $X$ over $R$} is defined to be
\[\Sigma^k(X;R)\coloneqq \Sigma^k(C_\ast(\widetilde{X};R);R), \]
where $\widetilde{X}$ is the universal cover of $X$, and $C_\ast(-;R)$ denotes the
cellular chain complex with coefficients in $R$.  For a group $G$ of type $\FP_n(R)$ we
define
\[\Sigma^k(G;R)\coloneqq \Sigma^k(C_\ast(EG;R);R).\]

The following theorem essentially combines Theorem~4, Proposition~3, and Corollaries~1
and~2 of~\cite{Farber-Geoghegan-Schuetz(2010)} but stated over a general ring, see
also~\cite[Theorem~2.16]{Hughes-Kielak(2024)}.

\begin{theorem}[Basic properties of homological BNSR invariants]\label{the:props_BNSR}
  Let $R$ be a ring.  Let $X$ be a connected CW complex with $C_\ast(X)$ finite $n$-type
  over $R$ and let $G=\pi_1(X)$.  The following conclusions hold:
  \begin{enumerate}
  \item $\Sigma^n(X;R)$ are open subsets of $S(G)$;
  \item If $\widetilde{X}$ is $n$-connected then
    \[\Sigma^n(X;R)=\Sigma^n(G;R) \text{ and
      }\Sigma^{n+1}(X;R)\subseteq\Sigma^{n+1}(G;R);\]
  \item If $X$ is finite, then for all $n\geq\dim X$ we have
    $\Sigma^n(X;R)=\Sigma^{\dim X}(X;R)$.
  \end{enumerate}
\end{theorem}

We record an adaptation of Fisher's result (Theorem~\ref{the:Fisher_F}) to the context of
CW-complexes.  Whilst the the statement does not make mention of BNSR invariants, both
Kielak and Fisher's arguments utilities them heavily (hence why have included the versions
for spaces above.)

\begin{addendum}\label{adden:Fisher}
  Let $F$ be a field and let $n \in \IZ_{\ge 1}$. Let $X$ be a CW complex with $C_\ast(X)$
  of finite $n$-type over $F$ and suppose $G=\pi_1(X)$ is a virtually $\RFRS$ group.  The
  following are equivalent:
  \begin{enumerate}
  \item\label{aden:Fisher_F:(1)} We have $b_p(\widetilde X;\cald_{F[G]}) = 0$ for
    $p \le n$;
  \item\label{aden:Fisher_F:(2)} There is a finite-index subgroup $H\subseteq G$
    corresponding to a finite cover $Y\to X$ and an epimorphism $\phi \colon H \to \IZ$
    corresponding to an infinite cyclic cover $\bar Y\to Y$ such that $C_\ast(Y;F)$ has finite $n$-type over $F$;

  \item\label{aden:Fisher_F:(3)} There is a finite-index subgroup $H' \subseteq G$
    corresponding to a finite cover $Y'\to X$ and an epimorphism $\phi' \colon H' \to \IZ$
    corresponding to an infinite cyclic cover $\bar Y\to Y'$ such that $b_k(\bar Y;F)< \infty$ holds for $k \le n$.
  \end{enumerate}
\end{addendum}
\begin{proof}
  This is essentially identical to~\cite[Theorem~6.14]{Fisher(2024)}
  replacing the free resolution of $F$ by free $F G$-modules with the
  cellular chain complex $C_\ast(\widetilde X)$ and taking into
  account Lemma~\ref{lem:Novikov_rings_and_chain_complex}.  Note that
  the openness of the homological BNSR invariants for spaces allows us
  to apply Kielak's theorem~\cite[Theorem~6.6]{Fisher(2024)} (for the
  original proof see~\cite[Theorem~5.2]{Kielak(2020fibring)}) as in
  Fisher's proof.
\end{proof}

We also give a formulation of Theorem~\ref{the:FIK_new} for CW complexes.

\begin{addendum}\label{aden:FIK_new}
  Let $X$ be a CW complex of finite $(n+1)$-type and suppose
  $G=\pi_1(X)$ is a virtually $\RFRS$ group.  Then there exists a
  finite set of primes $\calp_G$ with the following properties:

  \begin{enumerate}
  \item\label{adne:FIK_new:calp_G} The following statements are equivalent:
    \begin{enumerate}
    \item\label{aden:FIK_new:calp_G:vanishing}We have $b_k(X;\cald_{\IF_p[G]})=0$ for
      $k\le n$ and $p \in \calp_G$;
    \item\label{aden:FIK_new:calp_G:type}There exists a finite-index subgroup
      $H\subseteq G$ corresponding to a finite cover $Y\to X$ and an epimorphism $\phi\colon H \to \IZ$ 
      corresponding to an infinite cyclic cover $\bar Y\to Y$ such that $C_\ast(Y;F)$
	 has finite
      $n$-type over every field $F$;
    \end{enumerate}
  
  \item\label{aden:FIK_new:calp} Let $\calp$ be a any non-empty set of primes. Then the
    following statements are equivalent:
    \begin{enumerate}
    \item\label{aden:FIK_new:calp:vanishing} We have $b_k(X;\cald_{\IF_p[G]})=0$ for
      $k\le n$ and $p \in \calp$;
    \item\label{aden:FIK_new:calp:type} There exists a finite-index subgroup
      $H\subseteq G$ corresponding to a finite cover $Y\to X$ and an epimorphism $\phi\colon H \to \IZ$ 
      corresponding to an infinite cyclic cover $\bar Y\to Y$ such that $C_\ast(Y;F)$
	 has finite
      $n$-type over every field $F$ of characteristic $p$  satisfying $p \in \calp$  or $p \notin \calp_G$.
    \end{enumerate}
  \end{enumerate}
\end{addendum}
\begin{proof}
  The proof is essentially the same as~Theorem~\ref{the:FIK_new},
  making the obvious modifications to replace a free resolution of $F$ by $FG$-modules
  with the cellular chain complex $C_\ast(\widetilde X)$.
\end{proof}


\typeout{----- Section 3: General but not aspherical counterexamples --------------------}

\section{General but not aspherical counterexamples}%
\label{sec:General_but_not_aspherical_counterexamples}

\begin{theorem}\label{the:counterexamples_to_fibring_non_aspherical}
  Let $G$ be a group for which there exists is a connected finite $CW$-complex $B$ with
  fundamental group $G$ such that all $L^2$-Betti numbers
  $b_n^{(2)}(\widetilde{B};\caln(G))$ vanish or such that $G$ is a $\RALI$-group and for
  every field $F$ all $L^2$-Betti numbers $b_n^{(2)}(\widetilde{B};\cald_{F[G]})$ vanish.

  Then there exists for every $d \ge \max\{2\cdot \dim(B),6\}$ a connected closed smooth
  manifold $M$ with fundamental group $G$ and dimension $d$ with the following properties:

  \begin{enumerate}
  \item\label{the:counterexamples_to_fibring_non_aspherical:vanishing} All $L^2$-Betti
    numbers $b_n^{(2)}(\widetilde{M};\caln(G))$ vanish. If we additionally assume that $G$
    is a \RALI-group, we can arrange that for every field $F$ all $L^2$-Betti numbers
    $b_n^{(2)}(\widetilde{M};\cald_{F[G]})$ vanish;

  \item\label{the:counterexamples_to_fibring_non_aspherical:not_of-finite_type} For every
    finite covering $p \colon N \to M$ with connected total space $N$ and every
    epimorphism $\phi \colon \pi_1(N) \to \IZ$ the total space of the infinite cyclic
    covering $\overline{N} \to N$ associated to $\phi$ is not homotopy equivalent to a
    $CW$-complex of finite type;

  \item\label{the:counterexamples_to_fibring_non_aspherical:not_fibring} $M$ does not
    virtually fibre over $S^1$;
   
  \item\label{the:counterexamples_to_fibring_non_aspherical:fibred_over_all_fields} If
    additionally $G$ is virtually RFRS, then there exists a finite index subgroup
    $H\leqslant G$ with corresponding finite cover $N\to M$ and epimorphism
    $\phi:H\to \IZ$ with corresponding infinite cyclic cover $\bar N\to N$ such
    that $C_\ast(\bar N;F)$ is finite type.

  \end{enumerate}
\end{theorem}

Note that we are not claiming in
Theorem~\ref{the:counterexamples_to_fibring_non_aspherical} that $M$ is aspherical.
Moreover, the existence of the connected finite $CW$-complex $B$ appearing in
Theorem~\ref{the:counterexamples_to_fibring_non_aspherical} is necessary, since any
connected closed manifold has the homotopy type of a finite $CW$-complex.  Hence the class
of group appearing in Theorem~\ref{the:counterexamples_to_fibring_non_aspherical} is the
largest possible one to which the theorem can apply.

\subsection{Preparation for proving Theorem~\ref{the:counterexamples_to_fibring_non_aspherical}}
The proof of Theorem~\ref{the:counterexamples_to_fibring_non_aspherical} needs some
preparation.

Consider the following situation.  Let $\phi \colon G \to \IZ$ be a surjective group
homomorphism with kernel $K$.  Consider an element $y \in G$ which is mapped under
$\phi \colon G \to \IZ$ to a generator of $\IZ$.  Denote by $\gamma \colon K \to K$ the
automorphism of $K$ sending $k$ to $yky^{-1}$. Let $C_*$ be a $\IZ[K]$-chain complex.  Let
$f_* \colon C_* \to \gamma^*C_*$ be a $\IZ[K]$-chain map, where $\gamma^*C_*$ is the
$\IZ[K]$-chain complex obtained from $C_*$ by restriction with $\gamma $. We obtain a
$\IZ[G]$-chain map
\[
  \widehat{f_*} \colon \IZ[G] \otimes_{\IZ[K]} C_* \to \IZ[G] \otimes_{\IZ[K]} C_*
\]
by sending $u \otimes v$ for $u \in \IZ[G]$ and $v \in C_n$ to
$u \otimes v - uy^{-1} \otimes f_n(v)$.

\begin{definition}[Mapping torus for chain maps]\label{def:chain_complex_version_of_a_mapping_torus}
  Define the mapping torus $T(f_*)_*$ of $f_*$ to be the $\IZ[G]$-chain complex given by
  the mapping cone $\cone(\widehat{f_*})$ of $\widehat{f}_*$.
\end{definition}

\begin{lemma}\label{lem:elementary_properties_of_mapping_torus}
The following conclusions hold:
  \begin{enumerate}
  \item\label{lem:elementary_properties_of_mapping_torus:homotopy_invariance} Let $C_*$
    and $D_*$ be $\IZ[K]$-chain complexes.  Consider $\IZ[K]$-chain maps
    $f_* \colon C_* \to \gamma^*C_*$ and $g_* \colon D_* \to \gamma^*D_*$ and a
    $\IZ[K]$-chain homotopy equivalence $u_* \colon C_* \to D_*$ such that the following
    diagram of $\IZ[K]$-chain complexes commutes up to $\IZ[K]$-chain homotopy
    \[
      \xymatrix{C_* \ar[r]^-{f_*} \ar[d]_{u_*} & \gamma^* C_* \ar[d]^{\gamma^* u_*}
        \\
        D_* \ar[r]_-{g_*} & \gamma^* D_*.  }
    \]
    Then there is a $\IZ[G]$-chain homotopy equivalence
    \[
      v_* \colon T(f_*)_* \xrightarrow{\simeq} T(g_*)_*;
    \]

  \item\label{lem:elementary_properties_of_mapping_torus:starting_with_ZG-chain_complex}
    Let $C_*$ be a projective $\IZ[G]$-chain complex. Let
    $l(y)_* \colon i^*C_* \to \gamma^*i^*C_*$ be the $\IZ[K]$-chain map given by left
    multiplication with $y$, where $i^*C_*$ is the $\IZ[K]$-chain complex obtained from
    $C_*$ by restricting the $G$-action to $K$.

    Then there is a $\IZ[G]$-chain homotopy equivalence
    \[v_* \colon T(l(y)_*)_* \to C_*;
    \]

  \item\label{lem:elementary_properties_of_mapping_torus:L_upper_2-acyclic} Let $C_*$ be a
    projective $\IZ[G]$-chain complex.  Suppose that $i^*C_*$ is $\IZ[K]$-chain homotopy
    equivalent to a finitely generated free $\IZ[K]$-chain complex. Then $C_*$ is
    $\IZ[G]$-chain homotopy equivalent to a finitely generated free $\IZ[G]$-chain complex
    and all $L^2$-Betti numbers $b_n^{(2)}(C_*;\caln(G))$ vanish. If we additionally
    assume that $G$ is a \RALI-group, then for every field $F$ all $L^2$-Betti numbers
    $b_n^{(2)}(C_*;\cald_{F[G]})$ vanish.
  \end{enumerate}
\end{lemma}
\begin{proof}~\ref{lem:elementary_properties_of_mapping_torus:homotopy_invariance} This
  follows for instance from~\cite[Lemma~14.60 on page~521]{Lueck-Macko(2024)}.
  \\[1mm]~\ref{lem:elementary_properties_of_mapping_torus:starting_with_ZG-chain_complex}
  We obtain a $\IZ[G]$-chain map $u_* \colon \IZ[G] \otimes_{\IZ[K]} i^*C_* \to C_*$ by
  sending $g \otimes x$ to $gx$ for $g \in G$ and $x \in C_*$. Its composite with
  $\widehat{l(y)_*} \colon \IZ[G] \otimes_{\IZ[K]} i^*C_* \to \IZ[G] \otimes_{\IZ[K]}
  i^*C_*$ is trivial.  Hence we obtain a $\IZ[G]$-chain map
  \[
    v_* \colon T(l(y)_*)_* \to C_*
  \]
  from the universal property of the mapping cone, see for instance~\cite[Lemma~14.45 on
  page~515]{Lueck-Macko(2024)}.  There is an isomorphism of $\IZ[G]$-chain complexes
  \[
    a_* \colon \IZ[G] \otimes_{\IZ[K]} i^*C_* \xrightarrow{\cong} \IZ[\IZ] \otimes_{\IZ}
    C_*, \quad g \otimes x \mapsto \phi(g) \otimes gx,
  \]
  where $g' \in G$ acts on $\IZ[G] \otimes_{\IZ[K]} i^*C_*$ by
  $g'(g \otimes x) = (g'g) \otimes x$ for $g \in G$ and $x \in C_*$ and on
  $\IZ[\IZ] \otimes_{\IZ} C_*$ by $g' \cdot n \otimes x = (\phi(g') +n) \otimes g'x$ for
  $n \in \IZ$ and $x \in C_*$. Its inverse sends $n \otimes x$ to $g \otimes g^{-1}x$ for
  any choice $g \in G$ with $\phi(g) = n$. The following diagram of $\IZ[G]$-chain
  complexes commutes
  \[\xymatrix@!C=10em{\IZ[G] \otimes_{\IZ[K]} i^*C_* \ar[r]^-{\widehat{l(y)_*}}
      \ar[d]_{a_*}^{\cong} & \IZ[G] \otimes_{\IZ[K]} i^*C_* \ar[d]^{a_*}_{\cong}
      \\
      \IZ[\IZ] \otimes_{\IZ} C_* \ar[r]_-{r_{1 - \phi(y)} \otimes_{\IZ} \id_{C_*}}&
      \IZ[\IZ] \otimes_{\IZ} C_* }
  \]
  where $r_{1 - \phi(y)} \colon \IZ[\IZ] \to \IZ[\IZ]$ is given by multiplication with the
  element $1 - \phi(y) \in \IZ$ considered as an element in $\IZ[\IZ]$. Note that we have
  the short exact sequence of $\IZ[\IZ]$-modules
  $0 \to \IZ[\IZ] \xrightarrow{r_{1 -\phi(y)}} \IZ[\IZ] \xrightarrow{\epsilon} \IZ \to 0$
  and an obvious $\IZ[G]$-chain isomorphism
  $\IZ \otimes_{\IZ} C_* \xrightarrow{\cong} C_*$, where we consider $\IZ$ as a
  $\IZ[\IZ]$-module by the trivial $\IZ$-action and $\epsilon$ is the augmentation
  homomorphism. This implies that $v_*$ induces isomorphisms on the homology groups. Since
  the source and the target of $v_*$ are projective, $v_*$ is a $\IZ[G]$-chain homotopy
  equivalence.  \\[1mm]~\ref{lem:elementary_properties_of_mapping_torus:L_upper_2-acyclic}
  In view of
  assertions~\ref{lem:elementary_properties_of_mapping_torus:homotopy_invariance}
  and~\ref{lem:elementary_properties_of_mapping_torus:starting_with_ZG-chain_complex} it
  suffices to show for a finitely generated free $\IZ[K]$-chain complex $D_*$ and a
  $\IZ[K]$-chain map $w_* \colon D_* \to \gamma^*D_*$ that the $\IZ[G]$-chain complex
  $T(w_*)_*$ is $L^2$-acyclic. The proof of this fact is the obvious chain complex version
  of the proof of~\cite[Theorem~2.1]{Lueck(1994b)}.
\end{proof}

  \begin{lemma}\label{lem:units_in_Novikov_rings}
    Let $\Psi \colon S \xrightarrow{\cong} S$ be an automorphism of a ring $S$.
    \begin{enumerate}
     
    \item\label{lem:units_in_Novikov_rings:plus} Consider an element
      $x =\sum_{n = n_0} ^{\infty }s_n t^n \in S_{\Psi}((t))$ for some $n_0 \in \IZ$ such
      that $s_{n_0}$ is not a zero-divisor.

      Then $x$ is a unit in $S_{\Psi}((t))$ if and only if $s_{n_0}$ is a unit in $S$;

    \item\label{lem:units_in_Novikov_rings:minus} Consider an element
      $x =\sum_{n = -\infty} ^{n_0 }s_n t^n \in S_{\Psi}((t^{-1}))$ for some $n_0 \in \IZ$
      such that $s_{n_0}$ is not a zero-divisor.

      Then $x$ is a unit in $S_{\Psi}((t^{-1}))$ if and only if $s_{n_0}$ is a unit in
      $S$.

    \end{enumerate}
  \end{lemma}
  \begin{proof} We give the proof only for
    assertion~\ref{lem:units_in_Novikov_rings:plus}, the one for
    assertion~\ref{lem:units_in_Novikov_rings:minus} is completely analogous.
  
    If $s_0$ is a unit in $S$, then $x$ is a unit in $S_{\Psi}((t))$, since
    $t^{-n_0}s_0^{-1}$ and every element of the form $1 + s_1 t + s_2 t^2 + \cdots$ is a 
    unit in $S_{\Psi}((t))$.  Suppose that $x$ is a unit in $S_{\Psi}((t))$. Write the
    inverse $x^{-1} = \sum_{m = m_0} ^{\infty}s_m t^m$ for some $m_0 \in \IZ$ and
    $s_{m_0} \not = 0$.  Since $x \cdot x^{-1} = 1$ holds in $S_{\Psi}((t))$ and
    $s_{n_0} \cdot s_{m_0} \not= 0$ holds in $S$, we have $m_0 + n_0 = 0$ and
    $s_{n_0}\cdot s_{m_0} = 1$ in $S$. Analogously one proves $s_{m_0}\cdot s_{n_0} = 1$
    in $S$.  Hence $s_{n_0}$ is a unit.
  \end{proof}

   \begin{lemma}\label{lem:properties_of_C_ast}
     Let $G$ be a group and let $g \in G$ be any non-identity element.
     Consider the element $3g +2e \in \IZ[G]$ for $e \in G$ the identity.
     Let $C_*$ be the $1$-dimensional finite free $\IZ[G]$-chain
     complex whose first differential $c_1 \colon \IZ[G] \to \IZ[G]$
     is given by right multiplication with $3g+2e$.
     
     \begin{enumerate}
     \item\label{lem:properties_of_C_ast:Z} Consider any subgroup $H \subseteq G$ of
       finite index and any epimorphism $\phi \colon H \to \IZ$. Let
       $i \colon K = \ker(\phi) \to G$ be the inclusion.  Let $i^*C_*$ be the
       $\IZ[K]$-chain complex obtained from $C_*$ by restriction with $i$. Let
       $\gamma \colon K \xrightarrow{\cong} K$ be the automorphism given by conjugation
       with some preimage of a generator of $\IZ$ under $\phi \colon H \to \IZ$.  Let
       $j \colon H \to G$ be the inclusion.
      
       Then we have
       \[
         \IZ[H] = \IZ[K]_{\IZ[\gamma]}[t,t^{-1}] = \IZ[K]_{\IZ[\gamma]}((t)) \cap \IZ[K]_{\IZ[\gamma]}((t^{-1}))
       \]
      and both $H_0(\IZ[K]_{\IZ[\gamma]}((t)) \otimes_{\IZ[H]} j^*C_*)$ and
      $H_0(\IZ[K]_{\IZ[\gamma]}((t^{-1})) \otimes_{\IZ[H]} j^*C_*)$ are non-trivial.
      Moreover, the $\IZ[K]$-chain complex $i^* C_*$ is not $\IZ[K]$-chain homotopy
      equivalent to a finite projective $\IZ[K]$-chain complex;

    \item\label{lem:properties_of_C_ast:F} Consider any group epimorphism
      $\phi \colon G \to \IZ$ such that $\phi(g) \not= 0$.  Let $F$ be any field.

      Then $H_i(F[K]_{F[\gamma]}((t)) \otimes_{\IZ[G]} C_*)$ and
      $H_i(F[K]_{F[\gamma]}((t^{-1})) \otimes_{\IZ[G]} C_*)$ are trivial for all
      $i \in \IZ_{\ge 0}$.  Moreover, the $F[K]$-chain complex $F \otimes_{\IZ} i^*C_*$ is
      $F[K]$-chain homotopy equivalent to a finite projective $F[K]$-chain complex;

    \item\label{lem:properties_of_C_ast:vanishing_of_Betti_numbers} If there is a subgroup
      $G' \subseteq G$ of finite index such that $g \in G'$ holds and the image of $g$
      under the canoncial map $G' \to H_1(G')$ has infinite order, then all the
      $L^2$-Betti numbers $b_n^{(2)}(C_*;\caln(G))$ vanish. If we additionally assume that
      $G$ is a \RALI-group, then for every field $F$ all $L^2$-Betti numbers
      $b_n^{(2)}(C_*;\cald_{F[G]})$ vanish.

    \end{enumerate}
  \end{lemma}
  \begin{proof}~\ref{lem:properties_of_C_ast:Z} We begin with the case $H = G$. For every
    $k \in K$ none of the elements $3k$, $2e$, and $3k + 2e$ is a unit in $\IZ[K]$ or a
    zero-divisor in $\IZ[K]$, since this is true for their images under the augmentation
    map $\IZ[K] \to \IZ$.  Lemma~\ref{lem:units_in_Novikov_rings} implies that $3g +2e$ is
    not a unit in $\IZ[K]_{\IZ[\gamma]}((t^{\pm}))$. The element $3g +2e$ in
    $\IZ[K]_{\IZ[\gamma]}((t^{\pm}))$ is not a zero-divisor since its image under the
    obvious projection $\IZ[K]_{\IZ[\gamma]}((t^{\pm})) \to \IZ((t^{\pm}))$ induced by the
    augmentation homomorphism $\IZ[K] \to \IZ$ is not a zero-divisor.  Hence both
    $H_1(\IZ[K]_{\IZ[\gamma]}((t)) \otimes_{\IZ[G]} C_*)$ and
    $H_1(\IZ[K]_{\IZ[\gamma]}((t^{-1})) \otimes_{\IZ[G]} C_*)$ are trivial and both
    $H_0(\IZ[K]_{\IZ[\gamma]}((t)) \otimes_{\IZ[G]} C_*)$ and
    $H_0(\IZ[K]_{\IZ[\gamma]}((t^{-1})) \otimes_{\IZ[G]} C_*)$ are non-trivial.  Now apply
    Lemma~\ref{lem:Novikov_rings_and_chain_complex}.
  
    Next we treat the general case. Choose a set theoretic section $s$ of the canoncial
    projection $G \to H\backslash G$.  Then we obtain a $\IZ[H]$-isomorphism
    \[
      \alpha \colon \bigoplus_{y \in H\backslash G} \IZ[H] \xrightarrow{\cong} j^*\IZ[G],
      \quad \{h_y \mid y \in H \backslash G\} \mapsto \sum_{y \in H\backslash G} h_y \cdot
      s(y).
    \]
    Then there are a permutation isomorphism
    \[
     \tau \colon \bigoplus_{y \in H\backslash G} \IZ[H] \xrightarrow{\cong} \bigoplus_{y  \in H\backslash G} \IZ[H]
    \]
    coming from the bijection
    $H\backslash G \xrightarrow{\cong} H\backslash G$ induced by right multiplication with
    $g$ and elements $\{h_y \in H \mid y \in H \backslash G\}$ such that the following
    diagram of $\IZ[H]$-modules commutes
    \[
      \xymatrix@!C=10em{\bigoplus_{y \in H\backslash G} \IZ[H]
        \ar[r]^-{\tau}_-{\cong}\ar[d]_{\alpha}^{\cong} & \bigoplus_{y \in H\backslash G}
        \IZ[H] \ar[r]^-{\bigoplus_{y \in H\backslash G} r_{h_y}} & \bigoplus_{y \in
          H\backslash G} \IZ[H] \ar[d]^{\alpha}_{\cong}
        \\
        j^*\IZ[G] \ar[rr]_-{\res_G^H r_g} & & j^*\IZ[G] }
    \]
    where $r_g$ and $r_{h_y}$ stands for right multiplication with $g$ and $h_y$.  Hence
    we get
    \[
      H_0(\IZ[K]_{\IZ[\gamma]}((t^{\pm 1})) \otimes_{\IZ[H]} j^*C_*) = \bigoplus_{y \in H
        \backslash G} H_0(\IZ[K]_{\IZ[\gamma]}((t^{\pm 1})) \otimes_{\IZ[H]} D_*[y]).
    \]
    where $D_*[y]$ is the $1$-dimensional $\IZ[H]$-chain complex whose first differential
    $\IZ[H] \to \IZ[H]$ if given by right multiplication with $2h_y + 3e$.  Since we have
    already shown that $H_0(\IZ[K]_{\IZ[\gamma]}((t^{ \pm 1})) \otimes_{\IZ[H]} D_*[y])$
    is non-trivial for $y \in H\backslash G$, we conclude that
    $H_0(\IZ[K]_{\IZ[\gamma]}((t^{\pm 1})) \otimes_{\IZ[H]} j^*C_*)$ is non-trivial.  Now
    assertion~\ref{lem:properties_of_C_ast:Z} follows from
    Lemma~\ref{lem:Novikov_rings_and_chain_complex}.
    \\[1mm]\ref{lem:properties_of_C_ast:F} Let $F$ be a field. Then either $2$ and $3$ are
    units in $F$ or one of the elements $2$ and $3$ is a unit and the other is zero in
    $F$. Since $\phi(g) \not= \phi(e)$ holds, $2g + 3e$ considered as an element in
    $ F[G]$ is a unit in both $F[K]_{\IZ \psi}((t))$ and $F[K]_{\IZ \psi}((t^{-1}))$ by
    Lemma~\ref{lem:units_in_Novikov_rings}.  Now apply
    Lemma~\ref{lem:Novikov_rings_and_chain_complex}.
    \\[1mm]~\ref{lem:properties_of_C_ast:vanishing_of_Betti_numbers} As the $L^2$-Betti
    numbers $b_n^{(2)}(C_*;\caln(G))$ and $b_n^{(2)}(C_*;\cald_{F[G]})$ are compatible
    with induction for $G' \subset G$, see~\cite[Lemma~1.24~(4) on page~30]{Lueck(2002)}
    and~\cite[Theorem~3.12]{Avramidi-Lueck(2026)}, and $C_*$ is
    $\IZ[G] \otimes_{\IZ[G']} C_*'$ for the $1$-dimensional finite free $\IZ[G']$-chain
    complex $C_*'$ whose first differential $c_1 \colon \IZ[G'] \to \IZ[G']$ is given by
    right multiplication with $3g+2e$, we can assume without loss of generality that
    $G = G'$ and hence $b_1(G) \ge 1$ holds.  Since the image of $g$ under the canoncial
    map $G \to H_1(G)$ has infinite order, we can choose a group homomorphism
    $\phi \colon G \to \IZ$ satisfying $\phi(g) \not= 0$.  There is a subgroup of finite
    index $G'' \subseteq G$ such that $g$ is a generator of $\phi(G'')$.  Since
    $L^2$-Betti numbers are multiplicative under passing to subgroups of finite index,
    see~\cite[Theorem~1.12~(6) on page~22]{Lueck(2002)}
    and~\cite[Theorem~3.10]{Avramidi-Lueck(2026)}, we can assume without loss of
    generality that $\phi(g)$ is a generator of $\IZ$, otherwise replace $G$ by $G''$ and
    $\phi$ by $\phi|_{G''}$. Now
    assertion~\ref{lem:properties_of_C_ast:vanishing_of_Betti_numbers} follows from
    assertion~\ref{lem:properties_of_C_ast:F} and
    Lemma~\ref{lem:elementary_properties_of_mapping_torus}~%
\ref{lem:elementary_properties_of_mapping_torus:L_upper_2-acyclic}.
  \end{proof}

\begin{remark}\label{rem:Novikov_over_Z_and_F}
  Note that assertions~\ref{lem:properties_of_C_ast:Z} and~\ref{lem:properties_of_C_ast:F}
  of Lemma~\ref{lem:properties_of_C_ast} illustrate that the vanishing of the Novikov
  homology over $\IZ$ does not follow from the vanishing over every field $F$. This is in
  contrast to the fact that a finitely generated abelian group $A$ is trivial if and only
  if $\IF_p \otimes_{\IZ} A$ vanishes for every prime $p$.
\end{remark}

\subsection{Proof of Theorem~\ref{the:counterexamples_to_fibring_non_aspherical}}

Now we are ready to give the proof of
Theorem~\ref{the:counterexamples_to_fibring_non_aspherical}.

\begin{proof}[Proof of Theorem~\ref{the:counterexamples_to_fibring_non_aspherical}]
  We give the proof only for $L^2$-Betti numbers; the proof in the case, where $G$ is a
  \RALI-groups and we consider $L^2$-Betti numbers over $F[G]$ for any field $F$, is
  completely analogous.
  
  We begin with the hard case where there is a subgroup $G' \subseteq G$ of finite index
  with $b_1(G') >0$.  We can choose an element $g \in G'$ that the image of $g$ under the
  canoncial map $G' \to H_1(G')$ has infinite order.

  Let $C_*$ be the $\IZ[G]$-chain complex appearing in
  Lemma~\ref{lem:properties_of_C_ast}, i.e., $C_*$ is the $1$-dimensional finite free
  $\IZ[G]$-complex whose first differential $c_1 \colon \IZ[G] \to \IZ[G]$ is given by
  right multiplication with $3g+2e$.  We can attach to $B \vee S^2$ a $3$-cell $D^3$ such
  that for the resulting $CW$-complex $X$ the inclusion $B \to X$ is $2$-connected and
  yields therefore an identification $G = \pi_1(B) = \pi_1(X)$ and that there is an
  isomorphisms of $\IZ[G]$-chain complexes
  $C_*(\widetilde{B}) \oplus \Sigma ^2 C_* \xrightarrow{\cong} C_*(\widetilde{X})$.
  Lemma~\ref{lem:properties_of_C_ast}~\ref{lem:properties_of_C_ast:vanishing_of_Betti_numbers}
  implies that all the $L^2$-Betti numbers $b_n^{(2)}(C_*;\caln(G))$ vanish and, if $G$ is
  a \RALI-group, for every field $F$ all the $L^2$-Betti numbers
  $b_n^{(2)}(C_*;\cald_{F[G]})$ over $F[G]$ vanish.  For any subgroup $H \subseteq G$ of
  finite index and epimorphism $\phi \colon H \to \IZ$ the $\IZ[H]$-chain complex
  $\res_G^H C_*(\widetilde{X}) $ has non-trivial Novikov homology in dimension $2$ by
  Lemma~\ref{lem:properties_of_C_ast}~\ref{lem:properties_of_C_ast:Z}.  Consider any
  natural number $d$ with $d \ge 2 \cdot \dim(X)$. Since $\dim(X) \ge 3$, we have
  $d \ge 6$.  By taking the boundary $M = \partial Z$ of a regular neighborhood $Z$ of an
  embedding of $X$ into $\IR^{d+1}$, we get a closed smooth manifold $M$ of dimension $d$
  such that there is a $(d - \dim(X))$-connected map $f \colon M \to X$.  We have
  $d - \dim(X) \ge 3$.  In particular we get an identification
  $\pi_1(M) = \pi_1(Z) = \pi_1(X) = \pi_1(B) = G$.  The induced map
  $\widetilde{f} \colon \widetilde{M} \to \widetilde{X}$ is $(d - \dim(X))$-connected and
  hence $3$-connected.  Hence for every subgroup $H \subseteq G$ of finite index and
  epimorphism $\phi \colon H \to \IZ$ the second Novikov homology of
  $\res_G^H C_*(\widetilde{M})$ is non-trivial.  We conclude from
  Lemma~\ref{lem:criteria_for_finitely_dominated} and
  Lemma~\ref{lem:Novikov_rings_and_chain_complex} that for every finite covering
  $p \colon N \to M$ and every epimorphism $\phi \colon \pi_1(N) \to \IZ$ the total space
  of the infinite cyclic covering $\overline{N} \to N$ associated to $\phi$ is not
  homotopy equivalent to a $CW$-complex of finite type.

  This implies that $M$ does not virtually fibre over $S^1$.

  Since $Z$ is homotopy equivalent to $X$, all the $L^2$-Betti numbers
  $b_n^{(2)}(\widetilde{Z};\caln(G)) =b_n^{(2)}(C_*(\widetilde{Z});\caln(G))$ vanish. We
  conclude from Poincar\'e duality that all the $L^2$-Betti numbers
  $b_n^{(2)}(C_*(\widetilde{Z},\widetilde{\partial Z});\caln(G)) =
  b_n^{(2)}(C_*(\widetilde{Z},\widetilde{M});\caln(G))$ vanish, see~\cite[Theorem~1.35~(5)
  on page~37]{Lueck(2002)}. We conclude that all the $L^2$-Betti numbers
  $b_n^{(2)}(\widetilde{M};\caln(G)) =b_n^{(2)}(C_*(\widetilde{M});\caln(G))$ vanish by
  considering the long exact homology sequence.

  Next we deal with the easy case that there is no subgroup $G' \subseteq G$ of finite
  index with $b_1(G') > 0$.  By the construction above, but now applied to $B$ itself
  instead of $X$, we can construct a $d$-dimensional closed smooth manifold $M$ with
  fundamental group $G$ whose $L^2$-Betti numbers are all trivial.  Since there is no
  subgroup of finite index $G' \subseteq G$ for which there exists an epimorphism
  $\Phi \colon G' \to \IZ$, the manifold $M$ does not virtually fibre over $S^1$.

  We now prove
  Item\ref{the:counterexamples_to_fibring_non_aspherical:fibred_over_all_fields}.
   The existence of the finite index subgroup $H\leqslant G$ and character
  $\phi$ with the desired properties is given by combining
  Item~\ref{the:counterexamples_to_fibring_non_aspherical:vanishing} with
  Addendum~\ref{aden:FIK_new}. 
\end{proof}


\typeout{----- Section 4: Fibring and fibrations --------------------}

\section{Fibring and fibrations}\label{sec:fibring_and_fibrations}

Throughout this section let $F \xrightarrow{j} E \xrightarrow{f} B$ be a fibration of
spaces which have the homotopy type of a connected $CW$-complex. Recall that the
Leray--Serre spectral sequence for singular cohomology with coefficients in $\IQ$
converges to $H^{p+q}(E;\IQ)$ and has as $E^2$-term
\[E^{p,q}_2 =H_{\pi_1(B)}^p(B;H^q(F;\IQ)),\]
where the $\pi_1(B)$-action on $H^q(F;\IQ)$
comes from the fibre transport.  We denote by
\begin{equation}
  \tau \colon H^1(F;\IQ)^{\pi_1(B)} = H_{\pi_1(B)}^0(B;H^1(F;\IQ)) \to H^2(B;\IQ) = H_{\pi_1(B)}^2(B;H^0(F))
  \label{transgression_map_tau}
\end{equation}   
the \emph{transgression map} which is the differential $d_2^{0,1}$ 
of the second page starting at $(0,1)$. It fits into an exact sequence
\begin{multline}
  0 \to H^1(B;\IQ) \xrightarrow{H^1(f;\IQ)} H^1(E;\IQ) \xrightarrow{H^1(j;\IQ)}
  H^1(F;\IQ)^{\pi_1(B)}
  \\
  \xrightarrow{\tau} H^2(B;\IQ) \xrightarrow{H^2(f;\IQ)} H^2(E;\IQ).
  \label{exact_sequence_coming_fron_Leray-Serre}
\end{multline}

  \begin{lemma}\label{lem:transgression_and_H_upper_1(p;Q)}
    The following assertions are equivalent:
    \begin{enumerate}
    \item\label{lem:transgression_and_H_upper_1(p;Q):H_upper_1(f)_surjective} The map
      $H^1(f;\IQ) \colon H^1(B;\IQ) \to H^1(E;\IQ)$ is surjective;
    \item\label{lem:transgression_and_H_upper_1(p;Q):H_upper_1(f)_bijective} The map
      $H^1(f;\IQ) \colon H^1(B;\IQ) \to H^1(E;\IQ)$ is bijective;
    \item\label{lem:transgression_and_H_upper_1(p;Q):H_upper_1(j)} The map
      $H^1(j;\IQ) \colon H^1(E;\IQ) \to H^1(F;\IQ) $ is trivial;
    \item\label{lem:transgression_and_H_upper_1(p;Q):H_1_1(j)} The image of the map
      $H_1(j;\IZ) \colon H_1(F;\IZ) \to H_1(E;\IZ)$ is a torsion group;
    \item\label{lem:transgression_and_H_upper_1(p;Q):tau} The transgression map
      $\tau \colon H^1(F;\IQ)^{\pi_1(B)} \to H^2(B;\IQ)$ is injective.
    \end{enumerate}
  \end{lemma}
  \begin{proof} This follows from the exact
    sequence~\ref{exact_sequence_coming_fron_Leray-Serre} and the Universal Coefficient
    Theorem.
  \end{proof}
  
  Assume additionally that $\pi_1(j)$ is injective.  Let 
  $\phi_E \colon \pi_1(E) \to \IZ$ be an epimorphism, let $K_E$
  be the kernel of $\phi_E$, and let 
  $i_E \colon K_E \to \pi_1(E)$ be its inclusion. Then we obtain a commutative diagram of
  groups whose rows and columns are short exact sequence of groups:

\begin{equation}\label{diagram_with_exact_rows_and_columns}
  \xymatrix{& 1  \ar[d] & 1 \ar[d] & 1 \ar[d] &
    \\
    1 \ar[r] & K_F \ar[d]^{i_F} \ar[r]  & K_E \ar[d]^{i_E} \ar[r]  & K_B \ar[d]^{i_B} \ar[r]  & 1
    \\
    1 \ar[r]  & \pi_1(F) \ar[d]^{\phi_F}  \ar[r]^{\pi_1(j)}  & \pi_1(E) \ar[d]^{\phi_E} \ar[r]^{\pi_1(f)}
    & \pi_1(B) \ar[d]^{\phi_B}  \ar[r]  & 1
    \\
    1 \ar[r]  & Q_F  \ar[d] \ar[r]^{\alpha}  & \IZ \ar[d] \ar[r]^{\beta}  & Q_B \ar[d] \ar[r]  & 1
    \\
    & 1 & 1 & 1 &
  }
\end{equation}
if we put $K_F = \pi_1(j)^{-1}(K_E)$ and $K_B = \pi_1(f)(K_E)$, let $i_F$ and $i_B$ be the
obvious inclusions of normal subgroups, and put $Q_F = \pi_1(F)/K_F$ and
$Q_B = \pi_1(B)/K_B$. All the other maps appearing in the diagram are then obvious.

 \begin{lemma}\label{lem:order_of_Q_B_and_transgression}
   The following conclusions hold:

   \begin{enumerate}

   \item\label{lem:order_of_Q_B_and_transgression:transgression_injective} Suppose that
     there is an epimorphism $\phi_E \colon \pi_1(E) \to \IZ$.  If the transgression map
     $\tau \colon H^1(F;\IQ)^{\pi_1(B)} \to H^2(B;\IQ)$ is injective, then $Q_F$ is
     trivial and the map $\beta \colon \IZ \to Q_B$ is an isomorphism.

   \item\label{lem:order_of_Q_B_and_transgression:transgression_not_injective} If the
     transgression map $\tau \colon H^1(F;\IQ)^{\pi_1(B)} \to H^2(B;\IQ)$ is not
     injective, then there exists an epimorphism $\phi_E \colon \pi_1(E) \to
     \IZ$. Moreover, $Q_F$ is infinite cyclic and $Q_B$ is finite.

   \end{enumerate}
 \end{lemma}
 \begin{proof}~\ref{lem:order_of_Q_B_and_transgression:transgression_injective} Since the
   transgression map $\tau \colon H^1(F;\IQ)^{\pi_1(B)} \to H^2(B;\IQ)$ is injective by
   assumption, the image of the map $H_1(j;\IZ) \colon H_1(F;\IZ) \to H_1(E;\IZ)$ is a
   torsion group by Lemma~\ref{lem:transgression_and_H_upper_1(p;Q)}.  Since $\IZ$ is
   torsionfree, we conclude from the diagram~\eqref{diagram_with_exact_rows_and_columns}
   that the image of $\alpha$ is trivial.  Hence $Q_F$ is trivial and the map
   $\beta \colon \IZ \to Q_B$ is an isomorphism.
   \\[1mm]~\ref{lem:order_of_Q_B_and_transgression:transgression_not_injective} Since the
   transgression $\tau \colon H^1(F;\IQ)^{\pi_1(B)} \to H^2(B;\IQ)$ is not injective, the
   map $H^1(f;\IQ) \colon H^1(B;\IQ) \to H^1(E;\IQ)$ is not surjective by
   Lemma~\ref{lem:transgression_and_H_upper_1(p;Q)}.  This implies that
   $H^1(f;\IZ) \colon H^1(B;\IZ) \to H^1(E;\IZ)$ is not surjective.  Hence we can find a
   non-trivial homomorphism $\phi_E \colon \pi_1(E) \to \IZ$ such that there is no
   homomorphism $\mu \colon \pi_1(B) \to \IZ$ with $\mu \circ \pi_1(f) = \phi_E$. Since we
   can replace $\phi_E \colon \pi_1(E) \to \IZ$ by the induced group epimorphism
   $\phi_E \colon \pi_1(E) \to \im(\phi_E)$, we can assume without loss of generality that
   $\phi_E \colon \pi_1(E) \to \IZ$ is surjective and that there is no group homomorphism
   $\mu \colon \pi_1(B) \to \IZ$ with $\mu \circ \pi_1(f) = \phi_E$. Hence the map
   $\beta \colon \IZ \to Q_B$ appearing in
   diagram~\eqref{diagram_with_exact_rows_and_columns} cannot be bijective.  This implies
   that $Q_B$ has finite order and $Q_F$ is infinite cyclic.
 \end{proof}

 Consider an epimorphism $\phi_E \colon \pi_1(E) \to \IZ$.  Let
 $p_{\phi_F} \colon \overline{F}_{\phi_F} \to F$,
 $p_{\phi_E} \colon \overline{E}_{\phi_E} \to E$, and
 $p_{\phi_B} \colon \overline{B}_{\phi_B} \to B$, be the coverings with path connected
 total spaces associated to the epimorphisms $\phi_F$, $\phi_E$, and $\phi_B$. Note that
 these are regular coverings with $Q_F$, $Q_E$, and $Q_B$ as deck transformation groups.
 The given fibration $F \xrightarrow{j} E \xrightarrow{f} B$ induces a fibration of spaces
 which have the homotopy types of connected $CW$-complexes

 \begin{equation}
   \overline{F}_{\phi_F} \xrightarrow{\overline{j}} \overline{E} _{\phi_E}
   \xrightarrow{\overline{f}} \overline{B}_{\phi_B}.
   \label{induced_fibrations_on_the_coverings}
 \end{equation}

 The construction is left to the reader and is similar to the one appearing in the proof
 of
 Lemma~\ref{the:fibring_and_fibrations}~\ref{the:fibring_and_fibrations:d_injective_plus_extra}.

 \begin{lemma}\label{lem:fibrations_and_finite_type}
   Consider $d \in \IZ_{\ge 0} $. Let $F \xrightarrow{j} E \xrightarrow{f} B$ be a
   fibration. Then:

   \begin{enumerate}

   \item\label{lem:fibrations_and_finite_type:F_B_to_E_F_d} If both $F$ and $B$ are of
     type $\F_d$, then $E$ is of type $\F_d$;

   \item\label{lem:fibrations_and_finite_type:F_E_to_B_F_d} If $F$ is of type $\F_{d-1}$
     and $E$ is of type $\F_{d}$, then $B$ is of type $\F_{d}$;

   \item\label{lem:fibrations_and_finite_type:F_E_to_B_F_infty} If $F$ is of type
     $\F_{\infty}$ and $E$ is of type $\F_{\infty}$, then $B$ is of type $\F_{\infty}$;

   \item\label{lem:fibrations_and_finite_type:F_B_to_E_general} If both $F$ and $B$ are
     homotopy equivalent to finite $CW$-complexes, finitely dominated $CW$-complexes, or
     to finite dimen\-sional $CW$-complexes, then the same is true for $E$.
   \end{enumerate}
 \end{lemma}
 \begin{proof}
   Assertions~\ref{lem:fibrations_and_finite_type:F_B_to_E_F_d},~%
\ref{lem:fibrations_and_finite_type:F_E_to_B_F_d},
   and~\ref{lem:fibrations_and_finite_type:F_E_to_B_F_infty} are proved
   in~\cite[Lemma~7.2]{Lueck(1997a)}.  The proof of
   assertion~\ref{lem:fibrations_and_finite_type:F_B_to_E_general} is analogous to the one
   of assertion~\ref{lem:fibrations_and_finite_type:F_B_to_E_F_d}.
 \end{proof}

 Although we will not use the next lemma in the sequel, we record it as it may be of independent interest.

\begin{lemma}\label{lem:virtually_fibring_and_bundles_and_products_inheritiance_from_base_total-space}
  Let $F \to E \xrightarrow{p} B$ be a smooth fibre bundle of connected smooth closed
  manifolds.  Then
  \begin{enumerate}
  \item\label{lem:virtually_fibring_and_bundles_and_products_inheritiance_from_base_total-space:fibring}
    If $B$ fibres over $S^1$, then $E$ fibres over $S^1$;
  \item\label{lem:virtually_fibring_and_bundles_and_products_inheritiance_from_base_total-space:virtually_fibring}
    If $B$ virtually fibres over $S^1$, then $E$ virtually fibres over $S^1$.
  \end{enumerate}
\end{lemma}
\begin{proof}~\ref{lem:virtually_fibring_and_bundles_and_products_inheritiance_from_base_total-space:fibring}
  Let $D \to B \xrightarrow{q} S^1$ be a smooth fibre bundle of connected smooth closed
  manifolds. Then the composite $q \circ p \colon E \to S^1$ is a smooth fibre bundle of
  connected smooth closed manifolds, whose fibre $F'$ is the total space of a smooth
  closed fibre bundle $F \to F' \to D$.
  \\[1mm]~\ref{lem:virtually_fibring_and_bundles_and_products_inheritiance_from_base_total-space:virtually_fibring}
  Let $q \colon \overline{B} \to B$ is a finite covering with connected $\overline{B}$
  such that $\overline{B}$ fibres over $S^1$.  Consider the pullback
  \[
    \xymatrix{\overline{E} \ar[r]^{\overline{p}} \ar[d]_{\overline{q}} & \overline{B}
      \ar[d]^q
      \\
      E \ar[r]_{\overline{q}} & B.  }
  \]
  Then $\overline{E}$ is connected, $\overline{q}$ is a finite covering, and there is a
  smooth fibre bundle of connected smooth closed manifolds
  $F \to \overline{E} \xrightarrow{\overline{q}} \overline{B}$. Then $\overline{E}$ fibres
  over $S^1$ by
  assertion~\ref{lem:virtually_fibring_and_bundles_and_products_inheritiance_from_base_total-space:fibring}.
\end{proof}

\begin{definition}[(Virtually) fibring]\label{def:virtual_fibring}
  Let $X$ be a space of the homotopy type of a connected $CW$-complex. Consider
  $d \in \IZ_{\ge 1} \amalg \{\infty\}$.

  We say that $X$ \emph{$\F_d$-fibres}, \emph{$\FF$-fibres}, or \emph{$\FP$-fibres}
  respectively, if for some epimorphism $\phi \colon \pi_1(X) \to \IZ$ the total space
  $\overline{X}_{\phi}$ of the associated infinite cyclic covering
  $p_{\phi} \colon \overline{X}_{\phi}\to X $ is of type $\F_d$, $\FF$, or $\FP$
  respectively

  We say that $X$ \emph{virtually $\F_d$-fibres}, if for some finite covering
  $\widehat{X} \to X$ with a path connected space $\widehat{X}$ as total space the space
  $\widehat{X}$ $\F_d$-fibres. Define the notions of \emph{virtually $\FF$-fibres} and
  \emph{virtually $\FP$-fibres} analogously.

\end{definition}

   \begin{lemma}\label{lem:F_d_fibring_and_fibrations_and_products_inheritiance_from_base_total-space}
     Let $F \to E \xrightarrow{f} B$ be fibration of spaces having the homotopy type of a
     connected $CW$-complex.  Consider $d \in \IZ_{\ge 1} \amalg \{\infty\}$.  Suppose
     that $F$ is of type $\F_{d-1}$ using the convention $\infty -1 = \infty$.  Then:
     \begin{enumerate}
     \item\label{lem:F_d_fibring_and_fibrations_and_products_inheritiance_from_base_total-space:fibring}
       If $B$ $\F_d$ fibres, then $E$ $\F_d$-fibres:
     \item\label{lem:F_dfibring_and_fibrations_and_products_inheritiance_from_base_total-space:virtually_fibring}
       If $B$ virtually $\F_d$-fibres, then $E$ virtually $\F_d$-fibres.
     \end{enumerate}
   \end{lemma}
   \begin{proof}~\ref{lem:F_d_fibring_and_fibrations_and_products_inheritiance_from_base_total-space:fibring}
     Let $q_B \colon \overline{B} \to B $ be an infinite cyclic covering with path
     connected total space such that $\overline{B}$ is of type $\F_d$.  Consider the
     pullback
     \[
       \xymatrix{\overline{E} \ar[r]^{\overline{f}}\ar[d]_{\overline{p}} & \overline{B}
         \ar[d]^{} \\E \ar[r]_{f} & B.  }
     \]
     Then $\overline{p}$ is an infinite cyclic covering with a path connected total space
     and we have the fibration
     $F \to \overline{E} \xrightarrow{\overline{f}} \overline{B}$. We conclude from
     Lemma~\ref{lem:fibrations_and_finite_type}~\ref{lem:fibrations_and_finite_type:F_E_to_B_F_d}
     that $E$ is of type $\F_d$. Hence $E$ $\F_d$-fibres.
     \\[1mm]~\ref{lem:F_dfibring_and_fibrations_and_products_inheritiance_from_base_total-space:virtually_fibring}
     Let $q_B \colon \widehat{B} \to B$ be a finite covering of path connected spaces such
     that $\widehat{B}$ $\F_d$-fibres.  Consider the pullback
     \[
       \xymatrix{\widehat{E} \ar[r]^{\widehat{f}}\ar[d]_{q_E} & \widehat{B} \ar[d]^{q_B}
         \\
         E \ar[r]_{f} & B.  }
     \]
     Then $q_E$ is a finite covering with a path connected space as total space and we
     have the fibration $F \to \widehat{E} \xrightarrow{\widehat{f}} \widehat{B}$. Now
     apply
     assertion~\ref{lem:F_d_fibring_and_fibrations_and_products_inheritiance_from_base_total-space:fibring}.
   \end{proof}

   Next we want to state results where the conclusion is the other way around, namely, we
   want to get information about $B$ from $E$.

\begin{theorem}[Fibring and fibrations]\label{the:fibring_and_fibrations}
  Let $F \xrightarrow{j} E \xrightarrow{f} B$ be a fibration of spaces which have the
  homotopy type of a connected $CW$-complex. Assume that $\pi_1(j)$ is injective. Consider
  any $d \in \IZ_{\ge 0} \amalg \{\infty\}$. Then:

  \begin{enumerate}
  \item\label{the:fibring_and_fibrations:d_is_injective} Suppose that the transgression
    map $\tau \colon H^1(F;\IQ)^{\pi_1(B)} \to H^2(B;\IQ)$ is injective and $F$ is of type
    $\F_{d-1}$.  If $E$ $\F_d$-fibres, then $B$ $\F_d$-fibres;

  \item\label{the:fibring_and_fibrations:d_is_not_injective} Suppose that the
    transgression map $\tau \colon H^1(F;\IQ) ^{\pi_1(B)}\to H^2(B;\IQ)$ is not injective,
    that for every infinite cyclic covering $\overline{F} \to F$ with path connected total
    space $\overline{F}$ the space $\overline{F}$ is of type $\F_d$, and that $B$ is of
    type $\F_d$.  Then $E$ $\F_d$-fibres.

  \item\label{the:fibring_and_fibrations:d_injective_plus_extra} Suppose that $\pi_1(B)$
    operates trivially on $H^1(F;\IQ)$, that the transgression map
    $\tau \colon H^1(F;\IQ) \to H^2(B;\IQ)$ is injective, that for every finite covering
    $q \colon \widehat{F} \to F$ with connected total space $\widehat{F}$ the induced
    homomorphism $H^1(q;\IQ) \colon H^1(F;\IQ) \to H^1(\widehat{F};\IQ)$ is bijective, and
    that $F$ is of type $\F_{d-1}$ using the convention $\infty -1 = \infty$.  If $E$
    virtually $\F_d$-fibres, then $B$ virtually $\F_d$-fibres.

  \end{enumerate}
\end{theorem}
\begin{proof}~\ref{the:fibring_and_fibrations:d_is_injective} Suppose that $E$
  $\F_d$-fibres.  Let $\phi_E \colon \pi_1(E) \to \IZ$ be a surjective epimorphism such
  that $E_{\overline{\phi_E}}$ is of type $\F_d$.  Since the transgression map
  $\tau \colon H^1(F;\IQ) \to H^2(B;\IQ)$ is injective, $Q_F$ is trivial and the map
  $\beta \colon Q_B \to \IZ$ is an isomorphism by
  Lemma~\ref{lem:order_of_Q_B_and_transgression}~%
\ref{lem:order_of_Q_B_and_transgression:transgression_injective}.  The fibration
  of~\eqref{induced_fibrations_on_the_coverings} reduces to a fibration
  $F \to \overline{E} _{\phi_E}\to \overline{B}_{\phi_B}$ and
  $q_{\Phi_B} \colon \overline{B}_{\phi_B} \to B$ is an infinite cyclic covering. Since
  $F$ is of type $\F_{d-1}$ by assumption, we conclude from
  Lemma~\ref{lem:fibrations_and_finite_type}~\ref{lem:fibrations_and_finite_type:F_E_to_B_F_d}
  that $\overline{B}_{\phi_B}$ is of  type $\F_d$. Hence $B$ $\F_d$-fibres.
  \\[1mm]~\ref{the:fibring_and_fibrations:d_is_not_injective}.  We conclude from
  Lemma~\ref{lem:order_of_Q_B_and_transgression}~%
\ref{lem:order_of_Q_B_and_transgression:transgression_not_injective} that there exists a
  group epimorphism $\phi_E \colon \pi_1(E) \to \IZ$ and that $Q_F$ is infinite cyclic and
  $Q_B$ is finite.  Since $B$ is of type $\F_d$ and $\overline{B}_{\phi_B} \to B$ is a
  finite covering, $\overline{B}_{\phi_B}$ is of type $\F_d$. Since
  $\overline{F}_{\phi_F} \to F$ is an infinite cyclic covering of $F$ and hence
  $\overline{F}_{\phi_F}$ is of type $\F_d$ by assumption,
  Lemma~\ref{lem:fibrations_and_finite_type}~\ref{lem:fibrations_and_finite_type:F_B_to_E_F_d}
  applied to the fibration of~\eqref{induced_fibrations_on_the_coverings} shows that
  $\overline{E}_{\phi_E}$ is of type $\F_d$.  Hence $E$ $\F_d$-fibres.
  \\[1mm]~\ref{the:fibring_and_fibrations:d_injective_plus_extra} Let
  $q_E \colon \widehat{E} \to E$ be a finite covering of $E$ with path connected total
  space such that $\widehat{E}$ $\F_d$-fibres. Let $q_B \colon \widehat{B} \to B$ be the
  finite covering of $B$ with a path connected total space $\widehat{B}$ uniquely
  determined by the property that the image of
  $\pi_1(q_B) \colon \pi_1(\widehat{B}) \to \pi_1(B)$ agrees with the image of
  $\pi_1(f \circ q_E) \colon \pi_1(\widehat{E}) \to \pi_1(B)$.  Choose a map
  $\widehat{f} \colon \widehat{E} \to \widehat{B}$ with
  $q_B \circ \widehat{f} = f \circ q_E$. Consider the pullback
  \[
    \xymatrix{\overline{E} \ar[r]^{\overline{f}} \ar[d]_{\overline{q_B}} & \widehat{B}
      \ar[d]^{q_B}
      \\
      E \ar[r]_f & B.  }
  \]
  Because of the pullback property there is precisely one map
  $q_{\overline{E}} \colon \widehat{E} \to \overline{E}$ satisfying
  $\overline{f} \circ q_{\overline{E}} = \widehat{f}$ and
  $\overline{q_B} \circ \overline{q_E} = q_E$.  As $\overline{q_B} $ and $q_E$ are
  coverings, $q_{\overline{E}} \colon \widehat{E} \to \overline{E}$ is a covering.  We
  have the fibration
  $\overline{F} \xrightarrow{i} \overline{E} \xrightarrow{\overline{f}} \widehat{B}$ given
  by the pullback of $f$ with $q_B$.  Define a covering $q_F \colon \widehat{F} \to F$ by
  the pullback
  \[
    \xymatrix{\widehat{F} \ar[r]^{\overline{i}} \ar[d]_{q_F} & \widehat{E}
      \ar[d]_{q_{\overline{E}}}
      \\
      F \ar[r]_i & \overline{E}.  }
  \]
  Then $q_F \colon \widehat{F} \to F$ is a finite covering.  Since $F$, $\overline{E}$,
  and $\widehat{E}$ are path connected and
  $\pi_1(i) \colon \pi_1(F) \to \pi_1(\overline{E})$ is surjective, $\widehat{F}$ is path
  connected. Hence we get a fibration
  $\widehat{F} \xrightarrow{\widehat{j}} \widehat{E} \xrightarrow{\widehat{f}
  }\widehat{B}$ and a commutative diagram
  \[
    \xymatrix@!C=4em{\widehat{F} \ar[r]^{\widehat{i}} \ar[d]^{q_F} & \widehat{E}
      \ar[r]^{\widehat{f}} \ar[d]^{q_E} & \widehat{B} \ar[d]^{\id_{\widehat{B}}}
      \\
      F \ar[r]_i \ar[d]^{\id_F} & \overline{E} \ar[r]_{\overline{f}}
      \ar[d]^{\overline{q_B}} & \widehat{B} \ar[d]^{q_B}
      \\
      F \ar[r] & E \ar[r]^f & B.  }
  \]
  such that all spaces are path connected, all rows are fibrations, all vertical arrows
  are finite coverings, and the lower right and the upper left square are pullbacks.
  
  The map $H^1(q_F;\IQ) \colon H^1(F;\IQ) \to H^1(\widehat{F};\IQ)$ is bijective and
  $\pi_1(B)$ acts trivially on $H^1(F;\IQ)$ by assumption. Given
  $x \in \pi_1(\widehat{B})$, let $\overline{x}$ be its image under the injective map
  $\pi_1(q_B) \colon \pi_1(\widehat{B}) \to \pi_1(B)$ and we get a commutative diagram
  \[\xymatrix{H^1(F;\IQ) \ar[r]^{l_{\overline{x}}}\ar[d]_{H^1(q_F;\IQ)}^{\cong} &
      H^1(F;\IQ) \ar[d]^{H^1(q_F;\IQ)}_{\cong}
      \\
      H^1(\widehat{F};\IQ) \ar[r]^{l_x} & H^1(\widehat{F};\IQ) }
  \]
  where $l_{\widehat{x}}$ and $l_x$ come from the $\pi_1(B)$ and
  $\pi_1(\widehat{B})$-actions.  Hence $\pi_1(\widehat{B})$ acts trivially on
  $H^1(\widehat{F};\IQ)$.  If $\overline{\tau} \colon H^1(F;\IQ) \to H^1(\widehat{B},\IQ)$
  is the transgression map associated to the fibration
  $\F\to \overline{E} \xrightarrow{\overline{f}} \widehat{B}$ and
  $\widehat{\tau} \colon H^1(\widehat{F};\IQ) \to H^1(\widehat{B},\IQ)$ is the
  transgression map associated to the fibration
  $\widehat{F} \to \widehat{E} \xrightarrow{\widehat{f}} \widehat{B}$, then we get a
  commutative diagram
  \[
    \xymatrix{ H^1(F;\IQ) \ar[r]^-{\tau}\ar[d]_{\id_{H^1(F;\IQ)}}^{\cong} & H^1(B,\IQ)
      \ar[d]^{H^1(q_B;\IQ)}
      \\
      H^1(F;\IQ) \ar[r]_-{\overline{\tau}} \ar[d]_{\id_{H^1(q_F;\IQ)}}^{\cong} &
      H^1(\widehat{B},\IQ) \ar[d]^{\id_{H^1(\widehat{B};\IQ)}}
      \\
      H^1(\widehat{F};\IQ) \ar[r]_-{\widehat{\tau}} & H^1(\widehat{B},\IQ).  }
  \]
  Recall that $H^1(q_F;\IQ)$ is bijective by assumption.  The composite of the two right
  vertical arrows is injective, as $q_B$ is a finite covering.  The upper horizontal arrow
  is injective by assumption. Hence the lower vertical arrow $\widehat{\tau}$ is
  injective. Since $q_F$ is a finite covering and $F$ is of type $\F_{d-1}$ by assumption,
  $\widehat{F}$ is of type $\F_{d-1}$. Now
  assertion~\ref{the:fibring_and_fibrations:d_is_injective} applied to the fibration
  $\widehat{F} \to \widehat{E} \to \widehat{B}$ implies that $\widehat{B}$ $\F_d$-fibres.
  Hence $B$ virtually $\F_d$-fibres.
\end{proof}

\begin{remark}\label{rem:conditions_on_F}
  The condition on $F$ appearing in Theorem~\ref{the:fibring_and_fibrations}~%
\ref{the:fibring_and_fibrations:d_injective_plus_extra} that for every finite covering
  $c \colon \widehat{F} \to F$ with connected total space $\widehat{F}$ the induced
  homomorphism $H^1(c;\IQ) \colon H^1(F;\IQ) \to H^1(\widehat{F};\IQ)$ is bijective, is
  rather restrictive.  It is equivalent to the condition that for any subgroup
  $H \subseteq \pi_1(F)$ of finite index the map $H_1(BH;\IQ) \to H_1(B\pi_1(F);\IQ)$ is
  bijective.  It is satisfied if $\pi_1(F)$ is abelian.
  
  The condition on $F$ appearing in Theorem~\ref{the:fibring_and_fibrations}~%
\ref{the:fibring_and_fibrations:d_is_not_injective} that for every infinite cyclic
  covering $\overline{F} \to F$ with path connected total space $\overline{F}$ the space
  $\overline{F}$ is of type $\F_{d}$ seems to be even more restrictive.  It is satisfied
  if $\pi_1(F)$ is virtually finitely generated abelian and the singular homology
  $H_m(\widetilde{F};\IZ)$ is finitely generated as a $\IZ$-module for every
  $m \in \IZ_{\le d}$.
\end{remark}

Since we are mainly interested in the case that $F$ is an aspherical finite $CW$-complex
and any aspherical finite $CW$-complex with abelian fundamental group is homotopy
equivalent to $T^k$ for some $k \in \IZ_{\ge 0}$, the most interesting case for us is the
one of an orientable $T^k$-fibration.  Since up to fibre homotopy equivalence an
orientable $T^k$-fibration is a principal $T^k$-bundle, we will only consider principal
$T^k$-bundles in the sequel.


\typeout{----- Section 5: Aspherical counterexamples --------------------}


\section{Aspherical counterexamples}\label{sec:Aspherical_counterexamples}

Let $T^k \xrightarrow{j} E \xrightarrow{f} B$ be a principal $T^k$-bundle such that $B$ is
a closed smooth manifold.  Then $E$ is a closed smooth manifold of dimension $k + \dim(B)$
and $\pi_1(B)$ operates trivially on $H^1(T^k;\IQ)$.

Let $T^k \to ET^k \to BT^k$ be the universal principal $T^k$-bundle.  Given a principal
$T^k$-bundle $p \colon E \to B$ over a $CW$-complex $B$, there is a so called
\emph{classifying map} $c_p \colon B \to BT^k$ which is up to homotopy uniquely determined
by the property that the pullback of the universal principal $T^k$-bundle with $c_p$ to
$B$ is isomorphic to $p$.  Since the exact sequence
sequence~\eqref{exact_sequence_coming_fron_Leray-Serre} is natural in $B$ and $ET^k$ is
contractible, we obtain a commutative square of $\IQ$-modules
\[\xymatrix{H^1(T^k;\IQ) \ar[r]^{\tau}_{\cong} \ar[d]_{\id_{H^1(T^k;\IQ) }}^{\cong} &
    H^2(BT^k;\IQ) \ar[d]^{H^2(c_p;\IQ)}
    \\
    H^1(T^k;\IQ) \ar[r]_{\tau} & H^2(B;\IQ) }
\]
where the left vertical arrow and the upper horizontal arrows are bijective.  Hence the
transgression homomorphism $\tau \colon H^1(T^k;\IQ) \to H^2(B;\IQ)$ is injective if and
only if the homomorphism $H^2(c_p;\IQ) \colon H^2(BT^k;\IQ) \to H^2(B;\IQ)$ is injective.

Note for the sequel that a model for $BS^1 = BT^1$ is given by $\IC\IP^{\infty}$ and that
a model for the universal principal $T^k$-bundle $ET^k \to BT^k$ is given by the direct
product of $k$ copies of the universal $S^1$-principal bundle $ES^1 \to BS^1$.

For every infinite cyclic covering $\overline{F} \to T^k$ with path connected total space
$\overline{F}$ the space $\overline{F}$ is homotopy equivalent to $T^{k-1}$ and in
particular of type $\F_{\infty}$.  For any finite covering
$p \colon \widehat{T^n} \to T^n$ the induced map
$H^1(p;\IQ) \colon H^1(T^n;\IQ) \to H^1(\widehat{T^n};\IQ)$ is bijective.

  \begin{theorem}\label{the:fibring_and_principal_T_upper_n_bundles}
    Let $T^k \xrightarrow{j} M \xrightarrow{f} B$ be a principal $T^k$-bundle for
    $k \in \IZ_{\ge 1}$ such that $B$ is a closed smooth manifold of dimension $\dim(B)$
    and $\pi_1(j)$ is injective.

    Then:

    \begin{enumerate}

    \item\label{the:fibring_and_principal_T_upper_n_bundles:aspherical} If we
      additionally assume that $B$ is aspherical, then

      \begin{enumerate}
      \item\label{the:fibring_and_principal_T_upper_n_bundles:aspherical:apherical)} the
        manifold $M$ is an aspherical closed manifold of dimension $k + \dim(B)$

      \item\label{the:fibring_and_principal_T_upper_n_bundles:aspherical:pi_1(j)_injective}
        the map $\pi_1(j)$ is automatically injective.

      \end{enumerate}
    
    \item\label{the:fibring_and_principal_T_upper_n_bundles:vanishing_L_upper_Betti_numbers}
      The $L^2$-Betti number $b_n(\widetilde{M})$ vanishes for every $n \ge 0$;
        
    \item\label{the:fibring_and_principal_T_upper_n_bundles:vanishing_L_upper_F_p_Betti_numbers}
      If we additionally assume that $\pi_1(M)$ is a \RALI-group, then for every field $F$
      and $n \ge 0$ the $L^2$-Betti number $b_n^{(2)}(\widetilde{M};\cald_{F[G]})$
      vanishes:

    \item\label{the:fibring_and_principal_T_upper_n_bundles:B_F_d_fibres_implies_E_fibres}
      Suppose that the map $H^2(c_p;\IQ) \colon H^2(BT^k;\IQ) \to H^2(B;\IQ)$ is
      injective.  Consider any $d \in \IZ_{\ge 0}$.
      If $E$ $\F_d$-fibres, then $B$ $\F_d$-fibres.

    \item\label{the:fibring_and_principal_T_upper_n_bundles:E_fibres} Suppose that the
      map $H^2(c_p;\IQ) \colon H^2(BT^k;\IQ) \to H^2(B;\IQ)$ is not injective.
      Then $E$ $\F_{\infty}$-fibres;

    \item\label{the:fibring_and_principal_T_upper_n_bundles:B_F_d_virtually_fibres_implies_virtually_E_fibres}
      Suppose that the map $H^2(c_p;\IQ) \colon H^2(BT^k;\IQ) \to H^2(B;\IQ)$ is
      injective.
      If $E$ virtually $\F_d$-fibres, then $B$ virtually $\F_d$-fibres.

    \end{enumerate}

  \end{theorem}
  \begin{proof}~\ref{the:fibring_and_principal_T_upper_n_bundles:vanishing_L_upper_Betti_numbers}
    This follows from~\cite[Theorem~1.40 on page~42]{Lueck(2002)}.
    \\[1mm]~\ref{the:fibring_and_principal_T_upper_n_bundles:vanishing_L_upper_F_p_Betti_numbers}
    This follows from~\cite[Theorem~3.12~(iv) and
    Theorem~3.26~(iv)]{Avramidi-Lueck(2026)}.  \\[1mm] All the other assertions follow
    from Theorem~\ref{the:fibring_and_fibrations}.
  \end{proof}
  Theorem~\ref{the:fibring_and_principal_T_upper_n_bundles} leads to our favourite
  desired example.

\begin{theorem}\label{the:aspherical_counterexamples_of_T_upper_k-bundles_over_products}
  Consider $k,m \in \IZ$ with $k \ge 1$ and $m \ge 0$.  For $i = 1,2, \ldots, k+m$
  consider an aspherical closed manifold $B_i$ of even dimension $\dim(B_i)$ and
  $d_i \in \IZ_{\ge 0}$ such that
  $b_{d_i}^{(2)}(\widetilde{B_i}) = b_{d_i}^{(2)}(\widetilde{B_i};\caln(\pi_1(B_i)))$ is
  non-trivial and $H^2(B;\IQ)$ is non-trivial. (The Singer Conjecture predicts that this
  can only happen if $2d_i= \dim(B_i)$.)  Let $\pr_i \colon E_i \to B_i$ be any principal
  $S^1$-bundle such that its Euler class $e(\pr_i) \in H^2(B_i;\IZ)$ is sent to a
  nontrivial element under the change of coefficients map $H^2(B_i;\IZ) \to
  H^2(B_i;\IQ)$. Consider the $T^k$-principal bundle $p \colon E \to B$ given by
  \[
    p = \prod_{i = 1}^k p_i \times \prod_{j = k+1}^{k+m} \id_{B_i} \colon M = \prod_{i =
      1}^k E_i \times \prod_{j = k+1}^{k+m} B_j \to B = \prod_{k = 1}^{k+m} B_i.
  \]
  Then we get:

  \begin{enumerate}
  \item\label{the:aspherical_counterexamples_of_T_upper_k-bundles_over_products:apherical}
    The total space $M$ is an aspherical closed manifold and has dimension
    $k + \sum_{i = 1}^{k+m} \dim(B_i)$.

  \item\label{the:aspherical_counterexamples_of_T_upper_k-bundles_over_products:vanishing_L_upper_Betti_numbers}
    The $L^2$-Betti number $b_n(\widetilde{M})$ vanishes for every $n \ge 0$.

  \item\label{the:aspherical_counterexamples_of_T_upper_k-bundles_over_products:vanishing_L_upper_F_p_Betti_numbers}
    If we additionally assume that $\pi_1(M)$ is a \RALI-group, then for every field $F$
    the $L^2$-Betti numbers $b_n^{(2)}(\widetilde{M};\cald_{F[G]})$ vanishes for every
    $n \ge 0$.

  \item\label{the:aspherical_counterexamples_of_T_upper_k-bundles_over_products:E_does_not_F_d}
    The smooth manifold $M$ does not virtually $\F_{d}$-fibre for
    $d = \sum_{i = 1}^{k+m} d_i$ and in particular does not virtually fibre over $S^1$.
  \end{enumerate}
\end{theorem}
\begin{proof}
  From the K\"unneth formula for $L^2$-Betti numbers, see~\cite[Theorem~1.35~(4) on
  page~37]{Lueck(2002)}, we conclude that $b_d^{(2)}(\widetilde{B})$ is non-trivial for
  $d = \sum_{i = 1}^{k+m} d_i$. Hence $B$ does not $\F_d$-fibre by~\cite[Theorem~6.63 on
  page~270]{Lueck(2002)}. Now one easily checks using the K\"unneth formula that the
  assumptions appearing in Theorem~\ref{the:fibring_and_principal_T_upper_n_bundles}~%
\ref{the:fibring_and_principal_T_upper_n_bundles:B_F_d_virtually_fibres_implies_virtually_E_fibres}
  are satisfied. Hence
  Theorem~\ref{the:aspherical_counterexamples_of_T_upper_k-bundles_over_products} follows
  from Theorem~\ref{the:fibring_and_principal_T_upper_n_bundles}~%
\ref{the:fibring_and_principal_T_upper_n_bundles:aspherical},~%
\ref{the:fibring_and_principal_T_upper_n_bundles:vanishing_L_upper_Betti_numbers},%
\ref{the:fibring_and_principal_T_upper_n_bundles:vanishing_L_upper_F_p_Betti_numbers},
  and~\ref{the:fibring_and_principal_T_upper_n_bundles:B_F_d_virtually_fibres_implies_virtually_E_fibres}.
\end{proof}

  \begin{remark}\label{rem:neither_hyperbolic_nor_virtually_RFRS}
    The fundamental group $\pi_1(M)$ of the manifold $M$ appearing in
    Theorem~\ref{the:fibring_and_principal_T_upper_n_bundles} does not contain a subgroup
    $G$ of finite index such that $G$ is residually (torsionfree and virtually abelian).
    Note that this implies that $G$ is not residually (locally
    indicable and virtually abelian) and hence not a \RFRS-group
    by~\cite[Theorem~6.3]{Okun-Schreve(2024orders)}. In particular $\pi_1(M)$ is not
    virtually {\RFRS}.\hfill \\

   \noindent\emph{Proof of Remark.} Suppose that $\pi_1(M)$ contains a subgroup $G$ of finite index such that $G$ is
    residually (torsionfree and virtually abelian). Then we can find a finite covering
    $p \colon \widehat{M} \to M$ such that $\widehat{M}$ is path connected
    $\pi_1(\widehat{M})$ is residually (torsionfree and virtually abelian).  Consider any
    element $g \in \pi_1(\widehat{M})$ with $g \not= e$. Choose an epimorphism
    $\psi \colon \pi_1(\widehat{M}) \to Q$ for a group $Q$ such that $\psi(g) \not= e$
    holds and $Q$ is torsionfree and virtually abelian.
  
    Since $\pi_1(\widehat{M})$ is finitely generated, $Q$ is finitely generated. Hence
    there exists $d \in \IZ_{\ge 1}$ such that $Q$ contains $\IZ^d$ as a normal subgroup
    of finite index $n = [Q: \IZ^d]$.  Then $\psi(g^n)$ lies in $\IZ^d$ and is different
    from zero.  Choose a group homomorphism $\mu \colon \IZ^d \to \IZ$ such that
    $\mu(\psi(g^n)) \not= 0$ holds.  Hence $\phi = \mu \circ \psi \colon \pi_1(N) \to \IZ$
    is a group homomorphism with $\phi(g^n) \not = 0$.  By passing to
    $\phi \colon \pi_1(\widehat{M}) \to \im(\phi)$ we can arrange that $\phi$ is
    surjective. By inspecting the proof of
    Theorem~\ref{the:fibring_and_fibrations}~\ref{the:fibring_and_fibrations:d_injective_plus_extra},
    we can find a finite covering $\widehat{B} \to B$ with a path connected total space
    $\widehat{B}$ and a principal $T^k$-bundle
    $T^k \xrightarrow{i} \widehat{M} \xrightarrow{\widehat{p}} \widehat{B}$ whose
    transgression map is injective.  Lemma~\ref{lem:order_of_Q_B_and_transgression}~%
\ref{lem:order_of_Q_B_and_transgression:transgression_injective} implies that
    $\phi \colon \pi_1(\widehat{M}) \to \IZ$ factorizes over
    $\pi_1(\widehat{p}) \colon \pi_1(\widehat{M}) \to \pi_1(\widehat{B})$.  Hence the
    image of $\pi_1(i) \colon \pi_1(S^1) \to \pi_1(\widehat{M})$ is contained in the
    kernel of $\phi$. This implies $g^n$ is not contained in the image
    $\pi_1(i) \colon \pi_1(S^1) \to \pi_1(\widehat{M})$ which is an infinite cyclic group
    as $\widehat{B}$ is aspherical.  We conclude that $g$ is not contained in the image of
    $\pi_1(i) \colon \pi_1(S^1) \to \pi_1(\widehat{M})$.  Since $g \in \pi_1(\widehat{M})$
    was an arbitrary element with $g \not = e$, we get a contradiction. \hfill $\square$
  \end{remark}

  \begin{remark}\label{rem:main_heorems_applies}
    In order to get interesting examples from
    Theorem~\ref{the:aspherical_counterexamples_of_T_upper_k-bundles_over_products}, one
    needs examples for the manifolds $B_i$ appearing there.  They can be constructed as
    follows.
    
    Consider $m \in \IZ_{\ge 1}$ and closed orientable surfaces $S_1$, $S_2$, \ldots,
    $S_m$ of genus $\ge 2$.  Then $B = \prod_{i = 1}^m S_i$ is an aspherical closed
    manifold of even dimension $2m$. We conclude $b_m^{(2)}(\widetilde{B}) \not = 0$ from
    K\"unneth formula for $L^2$-Betti numbers, see~\cite[Theorem~1.35~(4) on
    page~37]{Lueck(2002)}, since $b_1^{(2)}(\widetilde{S_i}) \not= 0$ and
    $b_l^{(2)}(\widetilde{S_i}) = 0$ for $l \not = 1$ hold by~\cite[Example~1.36 on
    page~40]{Lueck(2002)}. Obviously $H^2(B;\IQ)$ is non-trivial.
  \end{remark}

\begin{theorem}\label{the:apherical_counterexamples_with_residually_nilpotent_fundamental_group}
  Consider $d \in \IZ_{\ge 3}$.  Then there is an aspherical closed manifold $M$ of
  dimension $d$ with the following properties:

  \begin{enumerate}

  \item\label{the:apherical_counterexamples_with_residually_nilpotent_fundamental_group:virt_nil}
    The fundamental group is residually (torsionfree and nilpotent);

  \item\label{the:the:apherical_counterexamples_with_residually_nilpotent_fundamental_group:vanishing_L_upper_Betti_numbers}
    The $L^2$-Betti number $b_n(\widetilde{M})$ vanishes for every $n \ge 0$;

  \item\label{the:apherical_counterexamples_with_residually_nilpotent_fundamental_group::vanishing_L_upper_F_p_Betti_numbers}
    For every field $F$ the $L^2$-Betti number $b_n^{(2)}(\widetilde{M};\cald_{F[G]})$
    vanishes for every $n \ge 0$;

  \item\label{the:apherical_counterexamples_with_residually_nilpotent_fundamental_group:E_does_not_F_d}
    If $d = 2m+1$ or $d = 2m+ 2$ for $m \in \IZ_{\ge 1}$, then the smooth manifold $M$
    does not virtually $\F_{m}$-fibre and in particular does not virtually fibre over
    $S^1$.

  \end{enumerate}
  
\end{theorem}
\begin{proof} We begin with the case $d \not = 4$.  Consider a principal $S^1$-bundle
  $q' \colon E' \to T^2$ whose Euler class $e(q)$ is a generator of $H^2(T^2;\IZ)$.  Let
  $S$ be any surface of genus $\ge 2$ and $f \colon S \to T^2$ be a map of degree $1$.
  Let $q \colon E \to S$ be the principal $S^1$-bundle given by the pullback of $q'$ with
  $f$.
  \[
    \xymatrix@!C=3em{E \ar[r]^-{\overline{f}} \ar[d]_{q} & E' \ar[d]^{q'}
      \\
      S \ar[r]_-{f} & T^2.  }
  \]
  Then the Euler class $e(q)$ of $q$ is a generator of $H^2(S;\IZ)$ and in particular its
  image under the change of coefficients map $H^2(S;\IZ) \to H^2(S;\IQ)$ is non-trivial.

  We obtain the following commutative diagram of groups with central extensions as rows
  \[
    \xymatrix@!C=4em{1 \ar[r] & \pi_1(S^1) \ar[d]^{\id_{\pi_1(S^1)}} \ar[r]^{\pi_1(i)} &
      \pi_1(E) \ar[d]^{\pi_1(\overline{f})} \ar[r]^{\pi_1(q)} & \pi_1(S) \ar[d]^{\pi_1(f)}
      \ar[r] & 1
      \\
      1 \ar[r] & \pi_1(S^1) \ar[r]^{\pi_1(i')} & \pi_1(E') \ar[r]^{\pi_1(q')} & \pi_1(T^2)
      \ar[r] & 1.  }
  \]
  where $i$ and $i'$ denote the inclusions.  The group $\pi_1(E')$ is torsionfree and
  nilpotent. It is actually the three-dimensional Heisenberg group. The injective group
  homomorphism $\pi_1(\overline{f} \circ i) = \pi_1(i') \colon \pi_1(S^1) \to \pi_1(E')$
  sends any element different from the unit to an element in $ \pi_1(E')$ different from
  the unit.  Hence any element different from the unit in the kernel of $\pi_1(q)$ is sent
  to a non-trivial element in $\pi_1(E')$.  Since $\pi_1(S)$ is known to be residually
  (torsionfree and nilpotent), $\pi_1(E)$ is residually (torsionfree and nilpotent).

  If $d = 2m+1$ for $m \ge 1$, consider the principal $S^1$-bundle
  \[
    p = q \times \prod_{i = 1} ^{m-1} \id_S \colon E \times \prod_{i = 1} ^{m-1} S \; \to
    \; S \times \prod_{i = 1} ^{m-1} S
  \]
  and, if $d = 2m + 2$ for $m \ge 2$, consider the the principal $T^2$-bundle
  \[
    p =q \times q \times \prod_{i = 1} ^{m-1} \id_S \colon E \times E \times \prod_{i = 1}
    ^{m-1} \id_S \colon E \times E \times\prod_{i = 1} ^{m-1} S \; \to \;S \times S \times
    \prod_{i = 1} ^{m-1} S.
  \]
  Since $\pi_1(E)$ and $\pi_1(S)$ are residually (torsionfree and nilpotent) and finite
  products of residually (torsionfree and nilpotent) groups are again residually
  residually (torsionfree and nilpotent), the total space of $p$ has in both cases a
  residually (torsionfree and nilpotent) fundamental group.  Now the
  Theorem~\ref{the:apherical_counterexamples_with_residually_nilpotent_fundamental_group}
  follows from
  Theorem~\ref{the:aspherical_counterexamples_of_T_upper_k-bundles_over_products} applied
  to $p$ if we take $k = 1$ or $k = 2$ and $B_i = S$ for $i = 1,2, \ldots, (m+1)$.

  The proof in dimension $4$ is slightly more complicated since we cannot just take
  principal $T^2$-bundle $E \to S$ over $S$ since such $E$ will virtually fibre over $S^1$
  by
  Theorem~\ref{the:fibring_and_fibrations}~\ref{the:fibring_and_fibrations:d_is_not_injective}.
  Since $H^2(E';\IZ)$ is isomorphic to $\IZ^2$ for the $S^1$-principal bundle
  $q' \colon E' \to T^2$ above, we can find a principal $S^1$-bundle
  $q'' \colon E'' \to E'$ whose Euler class is mapped under $H^2(E';\IZ) \to H^1(E';\IQ)$
  to a non-trivial element. We obtain a central extension
  $1 \to \IZ \to \pi_1(E'') \to \pi_1(E') \to 1$. Since $\pi_1(E')$ is torsionfree and
  nilpotent, $\pi_1(E'')$ is torsionfree and nilpotent.

  Consider the pullback
  \[
    \xymatrix{M \ar[r] ^{\overline{\overline{f}}} \ar[d]_{\overline{q''}} & E''
      \ar[d]^{q''}
      \\
      E \ar[r]_{\overline{f}} & E'.  }
  \]
  We obtain the following commutative diagram of groups with central extensions as rows
  \[
    \xymatrix@!C=4em{1 \ar[r] & \pi_1(S^1) \ar[d]^{\id_{\pi_1(S^1)}} \ar[r] & \pi_1(M)
      \ar[d]^{\pi_1(\overline{f})} \ar[r]^{\pi_1(\overline{\overline{f}})} & \pi_1(E'')
      \ar[d]^{\pi_1(\overline{q''})} \ar[r] & 1
      \\
      1 \ar[r] & \pi_1(S^1) \ar[r] & \pi_1(E) \ar[r]^{\pi_1(\overline{f})} & \pi_1(E')
      \ar[r] & 1.  }
  \]
  Since $\pi_1(E)$ is residually (torsionfree and nilpotent) and $\pi_1(E'')$ is
  torsionfree and nilpotent, we can argue as above to show that $\pi_1(M)$ is residually
  (torsionfree and nilpotent).  Now we can apply
  Theorem~\ref{the:fibring_and_principal_T_upper_n_bundles}~%
\ref{the:fibring_and_principal_T_upper_n_bundles:B_F_d_fibres_implies_E_fibres} twice,
  namely to $q \colon E \to S$ and then to $\overline{q''} \colon M \to E$, and conclude
  that $E$ does not $\F_1$-fibre and then that $M$ does not $\F_1$-fibre.
\end{proof}

Note that in the proof above in dimension $4$ the composite
$q \circ \overline{q''} \colon M \to S$ is a locally trivial $T^2$-bundle.  One easily
checks using Lemma~\ref{lem:transgression_and_H_upper_1(p;Q)} that its transgression map
is injective which implies that the $\pi_1(S)$-action on $H_1(T^2;\IQ)$ is non-trivial and
it is not a principal $T^2$-bundle.  This is consistent with
Theorem~\ref{the:fibring_and_fibrations}.

\begin{remark}
Here is an alternative construction of the nilmanifold $E''$ expressed as a non-principal
$T^2$-bundle over $T^2$ that plays a crucial role in the proof in the 4-dimensional case.
For $a,b,c,d\in \IR$, let $A(a,b,c,d)$ be the matrix
\[A(a,b,c,d):=\begin{pmatrix}
    1&a&c&d \\
    0&1&b&b(b-1)/2 \\
    0&0&1&b \\
    0&0&0&1 \end{pmatrix}, \] and for $R=\IZ$, $R=\IR$, let $\IG(R)$ be the set of all
matrices $A(a,b,c,d)$ for $a,b,c,d\in R$.  It is easily verified that $\IG(R)$ is a group,
whose abelianization is isomorphic to the additive group $R\times R$.  The kernel of the
map to the abelianization is the matrices of the form $A(0,0,c,d)$, which form a subgroup
also isomorphic to $R\times R$.  The nilpotent Lie group $\IG(\IR)$ is clearly
homeomorphic to $\IR^4$, and $\IG(\IZ)$ is a discrete subgroup.  Hence the coset space
$\IG(\IZ)\backslash \IG(\IR)=E''$ is an aspherical 4-manifold with fundamental group the
torsion-free nilpotent group $\IG(\IZ)$.  If we define matrices $\alpha=A(1,0,0,0)$,
$\beta=A(0,1,0,0)$, $\gamma=A(0,0,1,0)$ and $\delta=A(0,0,0,1)$, then $\IG(\IZ)$ is
generated by $\alpha$ and $\beta$.  To see this, the four elements together clearly
generate, but also note that $\gamma=[\alpha,\beta]$, and that $\delta=[\gamma,\beta]$,
where use the convention that $[g,h]=g^{-1}h^{-1}gh$.  The matrix $\delta$ generates the
infinite cyclic centre of $\pi_1(E'')$, while $\gamma$ and $\delta$ together generate the
free abelian commutator subgroup.  The fact that $\gamma$ is not central shows that $E''$
is a non-principal $T^2$-bundle over $T^2$.  A presentation for $\IG(\IZ)$ in terms of
$\alpha$ and $\beta$ is given by above is given by
\[\langle \alpha,\beta,\,\,:\,\, [[\alpha,\beta],\alpha]=1=
  [[[\alpha,\beta],\beta],\beta]\rangle.\]
\end{remark}

\begin{example}\label{exa:B-a_surface}
  Let $B$ be an orientable connected closed surface of genus $g \ge 2$. Then $B$ is
  hyperbolic and the $L^2$-Betti number $b_1^{(2)}(\widetilde{B})$ is $2g -2$ and hence
  non-zero.  Obviously $H^2(B;\IZ) \cong \IZ$. For any non-trivial $e \in H^2(B;\IZ)$ we
  obtain a principal $S^1$-bundle $p_e \colon M_e \to B$.  Note that $M_e$ is an
  orientable closed Seifert $3$-manifold whose geometry is $\widetilde{\SL_2(\IR)}$.
  Theorem~\ref{the:apherical_counterexamples_with_residually_nilpotent_fundamental_group}
  implies that all $L^2$-Betti numbers of $\widetilde{M_e}$ vanish, $\pi_1(M)$ is
  residually (torsionfree and nilpotent) and that $M_e$ does not virtually fibre. Actually
  Theorem~\ref{the:apherical_counterexamples_with_residually_nilpotent_fundamental_group}
  shows that for every subgroup $G \subseteq \pi_1(M_e)$ of finite index and every
  surjective map $\phi \colon G \to \IZ$ the kernel of $\phi$ is not finitely
  generated. Moreover, the fundamental group of $M_e$ cannot be {\RFRS}. This follows from
  of Kielak's Theorem~\ref{the:Kielak} or directly from
  Remark~\ref{rem:neither_hyperbolic_nor_virtually_RFRS}.
\end{example}

  \begin{remark}\label{rem:RFRS_cannot_ve_replaced_by_RFRS}
    Example~\ref{exa:B-a_surface} shows that the condition $\RFRS$ cannot be weakened to
    residually (torsionfree nilpotent) in Kielak's Theorem~\ref{the:Kielak}. Note that
    such a generalisation was alluded to in~\cite[pages~17--18]{Kielak(2025ICM)}, asked
    about in~\cite[Question~7.2]{Fisher-Klinge(2024)}, and conjectured
    in~\cite[Conjecture~1.4]{Escartin-Ferrer(2025)}.

    One can also conclude by inspecting the proof of
    Theorem~\ref{the:fibring_and_fibrations} that the condition $\RFRS$ cannot be
    weakened to residually (torsionfree and nilpotent) in Fisher's
    Theorems~\ref{the:Fisher_Q} and~\ref{the:Fisher_F}.  The details of the proof are left
    to the reader.  At least we mention that a standard spectral sequence argument using
    Serre classes becomes relevant, which replaces the application
    of~\cite[Lemma~7.2]{Lueck(1997a)} in the proof of
    Theorem~\ref{the:fibring_and_fibrations}.
  \end{remark}

\subsection{Detection of non-fibring by other coverings than the universal one}%
\label{subsec:Detection_of_non-fibring_by_other_coverings_than_the_universal_one}

One can detect the non-fibring in the situation of
Theorem~\ref{the:fibring_and_principal_T_upper_n_bundles} by $L^2$-Betti numbers if one
is willing to consider other coverings than the universal covering.  Note that
Theorem~\ref{the:counterexamples_to_fibring_aspherical_non_universal} is weaker than
Theorem~\ref{the:fibring_and_principal_T_upper_n_bundles} since in
Theorem~\ref{the:counterexamples_to_fibring_aspherical_non_universal} we can deal with
property $\F_m$ for a specific $m$ and we have no control over the $n$ appearing in
assertion~\ref{the:counterexamples_to_fibring_aspherical_non_universal:non_vanishing} of
Theorem~\ref{the:counterexamples_to_fibring_aspherical_non_universal}.

\begin{theorem}\label{the:counterexamples_to_fibring_aspherical_non_universal}
  Consider an aspherical closed manifold $B$ of even dimension such that $\chi(B)$ is
  non-trivial and $H^2(B;\IQ)$ is non-trivial. Let $\pr \colon M \to B$ be any principal
  $S^1$-bundle such that its Euler class $e(\pr) \in H^2(B;\IZ)$ is sent to a nontrivial
  element under the change of coefficients map $H^2(B;\IZ) \to H^2(B;\IQ)$. Then:

  \begin{enumerate}
  \item\label{the:counterexamples_to_fibring_aspherical_non_universal:aspherical} The
    manifold $M$ is aspherical and closed and has dimension $1 + \dim(B)$;
   
  \item\label{the:counterexamples_to_fibring_aspherical_non_universal:non_vanishing}
    Consider any finite covering $p \colon N \to M$ with connected total space $N$ and any
    epimorphism $\phi \colon \pi_1(N) \to \IZ$. Let $\overline{N} \to N$ be the infinite
    cyclic covering associated to $\phi$. Then $b_n^{(2)}(\overline{N};\caln(\IZ))$ is
    non-trivial for at least one $n \in \IZ_{\ge 0}$;
  
  \item\label{the:counterexamples_to_fibring_aspherical_non_universal:non_virtually_fibring}
    $M$ does not virtually fibre over $S^1$.
  \end{enumerate}
\end{theorem}
\begin{proof}~\ref{the:counterexamples_to_fibring_aspherical_non_universal:aspherical}
  This is clear from the homotopy long exact sequence for the fibration.
  \\[1mm]~\ref{the:counterexamples_to_fibring_aspherical_non_universal:non_vanishing} As
  explained in the proof of By inspecting the proof of
  Theorem~\ref{the:fibring_and_fibrations}, which in fact implies
  Theorem~\ref{the:fibring_and_principal_T_upper_n_bundles}, we can find a finite
  covering $q \colon C \to B$ such that $C$ is connected, a principal $S^1$ bundle
  $\overline{\pr} \colon N \to C$, and an epimorphism
  $\overline{\phi} \colon \pi_1(C) \to \IZ$ satisfying
  $\overline{\phi} \circ \pi_1(\overline{\pr}) = \phi$.  Let
  $f_C \colon \overline{C} \to C$ be the infinite cyclic covering of $C$ associated to
  $\overline{\phi}$.  Consider the pullback
  \[
    \xymatrix{\overline{N} \ar[r]^{\overline{\overline{\pr}}} \ar[d]_{f_N} & \overline{C}
      \ar[d]^{f}
      \\
      N \ar[r]_{\overline{\pr}} & C.  }
  \]
  Then $f_N \colon \overline{N} \to N$ is the infinite cyclic covering associated to
  $\phi \colon \pi_1(N) \to \IZ$ and
  $\overline{\overline{\pr}} \colon \overline{N} \to \overline{C}$ is principal
  $S^1$-bundle, whose homological Gysin sequence
  \[
    \cdots \to H_n(\overline{N}) \xrightarrow{\overline{\overline{\pr}}} H_n(\overline{C})
    \to H_{n-2}(\overline{C}) \to H_{n-1}(\overline{N})
    \xrightarrow{\overline{\overline{\pr}}} H_{n-1}(\overline{C}) \to
    H_{n-3}(\overline{C}) \to \cdots
  \]
  is an exact sequence of $\IZ[\IZ]$-modules. It induces an exact sequence of
  $S^{-1}\IZ[\IZ]$-modules
  \begin{multline*}
    \cdots \to S^{-1}\IZ[\IZ] \otimes_{\IZ[\IZ]} H_n(\overline{N}) \to S^{-1}\IZ[\IZ]
    \otimes_{\IZ[\IZ]} H_n(\overline{C}) \to S^{-1}\IZ[\IZ] \otimes_{\IZ[\IZ]}
    H_{n-2}(\overline{C})
    \\
    \to S^{-1}\IZ[\IZ] \otimes_{\IZ[\IZ]} H_{n-1}(\overline{N}) \to S^{-1}\IZ[\IZ]
    \otimes_{\IZ[\IZ]} H_{n-1}(\overline{C}) \to S^{-1}\IZ[\IZ] \otimes_{\IZ[\IZ]}
    H_{n-3}(\overline{C}) \to \cdots
  \end{multline*}
  Suppose that $b_n^{(2)}(\overline{N};\caln(\IZ))$ is trivial for every
  $n \in \IZ_{\ge 0}$.  Then $S^{-1}\IZ[\IZ] \otimes_{\IZ[\IZ]} H_n(\overline{N})$
  vanishes for every $n \in \IZ_{\ge 0}$. Hence we get for every $n \in \IZ_{\ge 0}$ an
  isomorphism of $S^{-1}\IZ[\IZ]$-modules
  \[
    S^{-1}\IZ[\IZ] \otimes_{\IZ[\IZ]} H_n(\overline{C}) \xrightarrow{\cong} S^{-1}\IZ[\IZ]
    \otimes_{\IZ[\IZ]} H_{n-2}(\overline{C})
  \]
  This implies $S^{-1}\IZ[\IZ] \otimes_{\IZ[\IZ]} H_n(\overline{C}) = 0$ for
  $n \in \IZ_{\ge 0}$.  Hence $\chi(C) = 0$ and therefore $\chi(B) = 0$. Since by
  assumption we have $\chi(B) \not= 0$, we get a contradiction. Hence
  $b_n^{(2)}(\overline{N};\caln(\IZ))$ is non-trivial for at least one
  $n \in \IZ_{\ge 0}$;
  \\[1mm]~\ref{the:counterexamples_to_fibring_aspherical_non_universal:non_virtually_fibring}
  This follows from
  assertion~\ref{the:counterexamples_to_fibring_aspherical_non_universal:non_vanishing}
  and~\cite[Theorem~6.63 on page~270]{Lueck(2002)}.
\end{proof}


  \typeout{-------------------------- Section 6: Products -------------------------}

  \section{Products}\label{sec:Products}

  The next Theorem~\ref{the:products_and_fibring} shows that it is in general not
  possible to construct counterexamples to fibring in higher dimensions from lower
  dimensions just by taking products.

   \begin{theorem}\label{the:products_and_fibring}
     For $i = 1,2$ consider any $d_i \in \IZ_{\ge 6}$ and any finitely presented group
     $G_i$ such that there exists a group extension
     $1 \to K_i \to G_i \xrightarrow{\phi_i} \IZ \to 1$ with finitely presented $K_i$.

     Then there exist connected closed smooth manifolds $M_1$ and $M_2$ satisfying:

     \begin{enumerate}
     \item\label{the:products_and_fibring:M_and_N:dimension_and_fundamental_group} We
       have $\dim(M_i) = d_i$ and $\pi_1(M_i) \cong G_i$ for $i = 1,2$;

     \item\label{the:products_and_fibring:not_F_2} Both $M_1$ and $M_2$ do not virtually
       $\F_2$-fibre;

     \item\label{the:products_and_fibring:finitely_fibres} The product $M_1 \times M_2$
       $\FP$-fibres;

     \item\label{the:products_and_fibring:M_times_N:fibres_over_S_upper_1} If $G_1$ and
       $G_2$ are torsionfree Farrell--Jones groups, then $M_1 \times M_2$ fibres over
       $S^1$.
     \end{enumerate}
   \end{theorem}

   Its proof needs some preparation.
   
   \subsection{Preparation for the proof of Theorem~\ref{the:products_and_fibring}}

   In the sequel we consider the group epimorphism $\phi$ given by the composite
   \begin{equation}
     \phi \colon G_1 \times G_2  = \pi_1(X) \times \pi_1(Y) =\pi_1(X \times Y)
     \xrightarrow{\Phi_1 \times \Phi_2}  \IZ \times \IZ
     \xrightarrow{\begin{pmatrix} 1 & 1  \end{pmatrix}} \IZ
     \label{choice:of_Phi}
   \end{equation}
   and denote by $K$ its kernel. Moreover we fix two distinct primes $l_1$ and $l_2$.

   We call a $R$-chain complex \emph{homotopically of type $\FF$ or $\FP$} respectively if
   it is $R$-chain homotopy equivalent to a finite free or finite projective respectively
   $R$-chain complex.

     \begin{lemma}\label{lem:types_and_exact_sequences}
       Let $0 \to U_* \to V_* \to W_*\to 0$ be an exact sequence of $R$-chain complexes.
       If any two of the $R$-chain complexes $U_*$, $V_*$, and $W_*$ 
       are of type $\FF$ or of type $\FP$ respectively,
       then all three  of them are of type $\FF$ or of type $\FP$ respectively.
     \end{lemma}

     \begin{proof}
       This follows from~\cite[Theorem~11.2 on page~212 and Lemma~11.6 on
       page~216]{Lueck(1989)}.
     \end{proof}

     \begin{lemma}\label{lem:six_out_of_four}
       Let $0 \to U[i]_* \to V[i]_* \to W[i]_* \to 0$ be a short exact sequence of
       $\IZ[G_i]$-chain complexes for $i = 1,2$. Suppose that the restriction of the four
       $\IZ[G_1 \times G_2]$-chain complexes $V[1]_* \otimes_{\IZ} V[2]_*$,
       $V[1]_* \otimes_{\IZ} W[2]_*$, $W[1]_* \otimes_{\IZ} V[2]_*$, and
       $W[1]_* \otimes_{\IZ} W[2]_*$ to $K$ is of type $\FF$ or $\FP$ respectively.
       Then the $\IZ[K]$-chain complex
       $\res_{G_1 \times G_2}^K U[1]_* \otimes_{\IZ} U[2]_*$ is of type type $\FF$ or
       $\FP$ respectively.
     \end{lemma}
     \begin{proof}
       We obtain short exact sequence of $\IZ[G_1 \times G_2]$-chain complexes.
       \begin{eqnarray*}
         &
           0 \to U[1]_* \otimes_{\IZ} U[2]_* \to U[1]_* \otimes_{\IZ} V[2]_* \to U[1]_* \otimes_{\IZ} W[2]_* \to 0;
         &
         \\
         &
           0 \to U[1]_* \otimes_{\IZ} V[2]_* \to V[1]_* \otimes_{\IZ} V[2]_* \to W[1]_* \otimes_{\IZ} V[2]_* \to 0;
         &
         \\
         &
           0 \to U[1]_* \otimes_{\IZ} W[2]_* \to V[1]_* \otimes_{\IZ} W[2]_* \to W[1]_* \otimes_{\IZ} W[2]_* \to 0.
         &
       \end{eqnarray*}
       The stay exact after applying $\res_{G_1 \times G_2}^K$. Now the claim follows from an
       iterated application of Lemma~\ref{lem:types_and_exact_sequences}.
     \end{proof}
     
     For $i = 1,2$ we call a $\IZ[G_i]$-chain complex $B[i]_*$ \emph{special} if it is
     $\IZ[G_i]$-chain homotopy equivalent to
     $D[i]_* \oplus (\IZ[G_i] \otimes_{\IZ} E[i]_*)$ for a finite free $\IZ[G_i]$-chain
     complex $D[i]_*$ for which $\res_{G_i}^{K_i}$ is homotopically of type $\FF$ and a
     finite free $\IZ$-chain complex $E_*$ such that $H_n(E_*)$ is $l_i$-torsion.

     \begin{lemma}\label{lem:products_of_special_complexes}
       Let $B[i]_*$ be a special $\IZ[G_i]$-chain complex for $i = 1,2$.
        Then the $\IZ[K]$-chain complex
       $\res_{G_1 \times G_2}^K B[1]_* \otimes_{\IZ} B[2]_*$, which is obtained from the
       $\IZ[G_1 \times G_2] = \IZ[G_1] \otimes_{\IZ} \IZ[G_2]$-chain complex
       $B[1]_* \otimes_{\IZ} B[2]_*$ by restriction, is homotopically of type $\FF$.
     \end{lemma}
     \begin{proof}
       Let $D[i]_*$ be a free $\IZ[G_i]$-chain complex such that $\res_{G_i}^{K_i} D[i]_*$
       is homotopically of type $\FF$. Let $E[i]_*$ be a finite free $\IZ$-chain complex
       whose homology is $l_i$-torsion.  Then we get the following four
       $\IZ[G_1 \times G_2]$-chain complexes
       \begin{eqnarray*}
         & D[1]_* \otimes_{\IZ} D[2]_* ;&
         \\
         & D[1]_* \otimes_{\IZ} (\IZ[G_2]_* \otimes_{\IZ} E[2]_*);&
         \\
         & (\IZ[G_1]_* \otimes_{\IZ} E[1]_*) \otimes_{\IZ} D[2]_*; &
         \\
         & (\IZ[G_1]_* \otimes_{\IZ} E[1]_*) \otimes_{\IZ} (\IZ[G_2]_* \otimes_{\IZ} E[2]_*). &       
       \end{eqnarray*}
       We have to show that for each of them the restriction to $K$ is homotopically of
       type $\FF$.

       We conclude from Lemma~\ref{lem:elementary_properties_of_mapping_torus}~%
\ref{lem:elementary_properties_of_mapping_torus:starting_with_ZG-chain_complex}
       that $\res_{G_1 \times G_2}^K D[1]_* \otimes_{\IZ} D[2]_*$ is $\IZ[K]$-chain
       homotopy equivalent to $T(f_*)$ for some $\IZ[K_1 \times K_2]$-chain homotopy
       equivalence
       \[
         f_* \colon \res_{G_1}^{K_1}D[1]_* \otimes_{\IZ} \res_{G_2}^{K_1}D[2]_* \to
         \gamma^* \left(\res_{G_1}^{K_1}D[1]_* \otimes_{\IZ}
           \res_{G_2}^{K_1}D[2]_*\right).
       \]
       Since the $\IZ[K_i]$-chain complex $\res_{G_i}^{K_i}D[i]_*$ is homotopically of
       type $\FF$ by assumption, the $\IZ[K_1 \times K_2]$-chain complex
       $\res_{G_1}^{K_1}D[1]_* \otimes_{\IZ} \res_{G_2}^{K_1}D[2]_*$ is homotopically of
       type $\FF$. We conclude from
       Lemma~\ref{lem:elementary_properties_of_mapping_torus}~%
\ref{lem:elementary_properties_of_mapping_torus:homotopy_invariance} that the
       $\IZ[K]$-chain complex $T(f_*)$ and hence the $\IZ[K]$-chain complex
       $\res_{G_1 \times G_2}^K D[1]_* \otimes_{\IZ} D[2]_*$ are homotopically of type
       $\FF$.

       We have the $\IZ[G_1 \times G_2]$-chain isomorphism
       \[
         \bigl(\IZ[G_1 \times G_2] \otimes_{\IZ[G_1]} D[1]_*\bigr) \otimes_{\IZ} E[2]_*
         \xrightarrow{\cong} D[1]_* \otimes_{\IZ} (\IZ[G_2] \otimes E_*[2])
       \]
       sending $(g_1,g_2) \otimes x \otimes y$ to $g_1x \otimes g_2 \otimes y$.  Since
       $E[2]_*$ is a finite free $\IZ$-chain complex, it suffices to show that
       $\res_{G_1 \times G_2}^K \IZ[G_1 \times G_2] \otimes_{\IZ[G_1]} D[1]_*$ is homotopy
       finite. Since $(G_1 \times \{1\})\backslash (G_1 \times G_2)/K$ is trivial, we
       conclude from the Double Coset Formula applied to the subgroups $G_1 \times \{1\}$
       and $K$ of $G_1 \times G_2$ using the obvious identifications
       $K_1 = K_1 \times \{1\}$ and $G_1 = G_1 \times \{1\}$ that the $\IZ[K]$-chain map
       \[\IZ[K] \otimes_{\IZ[\{K_1 \times \{1\}]} \res_{G_1}^{K_1} D[1]_*
         \xrightarrow{\cong} \res_{G_1 \times G_2}^K \IZ[G_1 \times G_2]
         \otimes_{\IZ[G_1]} D[1]_*
       \]
       sending $k \otimes x$ to $k \otimes x$ is an isomorphism.  Hence it suffices to
       show that the $\IZ[K]$-chain complex
       $\IZ[K] \otimes_{\IZ[K_1]} \res_{G_1}^{K_1} D[1]_*$ is homotopically of type $\FF$.
       This follows from the assumption that $\res_{G_1}^{K_1} D[1]_*$ is homotopically of
       type $\FF$.

       The proof for
       $(\IZ[G_1]_* \otimes_{\IZ} E[1]_*) \otimes_{\IZ} D[2]_* \otimes_{\IZ}$ is
       analogous.
        
       Since the primes $l_1$ and $l_2$ are distinct, $H_n(E[1]_* \otimes_{\IZ} E[2]_*)$
       vanishes for all $n \in \IZ_{\ge 0}$.  Since $E[1]_*$ and $E[2]_*$ are free
       $\IZ$-chain complexes, the $\IZ$-chain complex $E[1]_* \otimes_{\IZ} E[2]_*$ is
       contractible.  We conclude that the $\IZ[G_1 \times G_2]$-chain complex
       $\IZ[G_1 \times G_2] \otimes_{\IZ} (E_*[1]) \otimes_{\IZ} E[2]_*$ is contractible.
       Hence the $\IZ[K]$-chain complex
       $\res_{G_1 \times G_2}^K (\IZ[G_1] \otimes E_*[1]) \otimes_{\IZ} (\IZ[G_2] \otimes
       E_*[2])$ is contractible and in particular homotopically of type $\FF$.
     \end{proof}

     Recall from~\cite[Subsection~5.6.1]{Lueck-Macko(2024)} that for a ring $R$ with
     involution $r \mapsto \overline{r}$ the \emph{dual $C^{d-*}$} of a $d$-dimensional
     finite projective $R$-chain complex $C_*$ is the $d$- dimensional finite projective
     $R$-chain complex whose $i$-th chain module is $\hom_R(C_{d-i},R)$, where the
     involution is used to define a left $R$-module structure on $\hom_R(C_{d-i},R)$ by
     the formula $(rf)(x) = f(x)\cdot \overline{r}$.  If $M$ is a connected compact
     manifold with first Stiefel Whitney class $w \colon \pi = \pi_1(M) \to \{\pm 1\}$,
     then the integral group ring $\IZ[\pi]$ is always equipped with \emph{the $w$-twisted
       involution} given by 
     \[
       \sum_{g \in \pi} \lambda_g \cdot g \mapsto \sum_{g \in \pi} \lambda_g \cdot w(g) \cdot g^{-1}.
      \]

\begin{lemma}\label{lem:FP_and_duals}
  Consider $i \in \{1,2\}$. Let $D[i]_*$ be a finite free $\IZ[G_i]$-chain complex of
  dimension $\le d_i + 1$ such $\res_{G_i}^{K_i} D[i]_*$ is homotopically of type
  $\FP$. Let $D[i]^{d_i+1 -*}$ be the dual $\IZ[G_i]$-chain complex of $D[i]_*$.
  Then the $\IZ[K_i]$-chain complex $\res_{G_i}^{K_i} D[i]^{d_i+1 -*}$ is homotopically of
  type $\FP$.
\end{lemma}
\begin{proof}
  We conclude from Lemma~\ref{lem:Novikov_rings_and_chain_complex} in the setup of
  Example~\ref{exa:Novikov_rings_and_group_rings} for $\phi_i \colon G_i \to \IZ$ that
  both $H_n\bigl(\IZ[K]_{\Psi}[[t]] \otimes_{\IZ[K]_{\Psi}[t,t^{-1}]} D[i]_*\bigr)$ and
  $H_n\bigl(\IZ[K]_{\Psi}[[t^{-1}]] \otimes_{\IZ[K]_{\Psi}[t,t^{-1}]} D[i]_*\bigr)$ vanish
  for every $n \in \IZ_{\ge 0}$.  Lemma~\ref{lem:Novikov_rings_and_chain_complex} implies
  that it suffices to show that both
  $H_n\bigl(\IZ[K]_{\Psi}[[t]] \otimes_{\IZ[K]_{\Psi}[t,t^{-1}]}D[i]^{d_i+1 -*}\bigr)$ and
  $H_n\bigl(\IZ[K]_{\Psi}[[t^{-1}]] \otimes_{\IZ[K]_{\Psi}[t,t^{-1}]}D[i]^{d_i+1
    -*}\bigr)$ vanish for every $n \in \IZ_{\ge 0}$. Hence it suffices to show that the
  Novikov homology
  $H_n\bigl(\IZ[K]_{\Psi}[[t]] \otimes_{\IZ[K]_{\Psi}[t,t^{-1}]}D[i]^{d_i+1 -*}\bigr)$
  vanishes for all $n \in \IZ_{\ge 0}$ provided that the Novikov homology
  $H_n\bigl(\IZ[K]_{\Psi}[[t^{-1}]] \otimes_{\IZ[K]_{\Psi}[t,t^{-1}]} D[i]_*\bigr)$
  vanishes for every $n \in \IZ_{\ge 0}$.

  Now some care is necessary since the involution on $\IZ[K]_{\Psi}[t,t^{-1}]$ does not
  extend to an involution $\IZ[K]_{\Psi}[[t]]$ because of $\overline{t}= \pm t^{-1}$.
  However, there is an anti ring homomorphism
  $\ast \colon \IZ[K]_{\Psi}[[t]] \to \IZ[K]_{\Psi}[[t^{-1}]]$ satisfying
  $\ast(au) = \overline{u} \cdot \ast(a)$ for $a \in \IZ[K]_{\Psi}[[t]]$ and
  $u \in \IZ[K]_{\Psi}[[t^{-1}]]$.

  We obtain a natural isomorphism of abelian groups
  \begin{multline*}
    \Gamma_i \colon \IZ[K]_{\Psi}[[t]] \otimes_{\IZ[K]_{\Psi}[t,t^{-1}]}
    \hom_{\IZ[K]_{\Psi}[t,t^{-1}]}(D[i]_*, \IZ[K]_{\Psi}[t,t^{-1}])
    \\
    \xrightarrow{\cong} \hom_{\IZ[K]_{\Psi}[[t^{-1}]]} (\IZ[K]_{\Psi}[[t^{-1}]]
    \otimes_{\IZ[K]_{\Psi}[t,t^{-1}]} D[i]_*, \IZ[K]_{\Psi}[[t^{-1}]])
  \end{multline*}
  by sending $a \otimes f$ to the $\IZ[K]_{\Psi}[[t^{-1}]]$-map from
  $\IZ[K]_{\Psi}[[t^{-1}]] \otimes_{\IZ[K]_{\Psi}[t,t^{-1}]} D[i]_*$ to
  $\IZ[K]_{\Psi}[[t^{-1}]]$ which sends $b \otimes x$ to $b\cdot f(x)\cdot
  \ast(a)$. Obviously $\Gamma_i(a \otimes f)$ is a $\IZ[K]_{\Psi}[[t^{-1}]]$-map. One has
  to check that this is compatible with the tensor relation on the source. This follows
  from the following computation for $u \in \IZ[K]_{\Psi}[[t]]$:
  \[
    \Gamma_i(au \otimes f)(b \otimes x) = b\cdot f(x)\cdot \ast(au) = b\cdot f(x)\cdot
    \overline{u}\cdot \ast(a) = b\cdot (uf)(x)\cdot \ast(a) = \Gamma_i(a \otimes uf)(b
    \otimes x).
  \]
  One easily checks that $\Gamma_i$ is an isomorphism of abelian groups, since $D[i]_*$ is
  a finitely generated free $\IZ[K]_{\Psi}[t,t^{-1}]$-module. Since $\Gamma_i$ is natural,
  we get an isomorphism of $\IZ$-chain complexes
  \begin{multline*}
    \Gamma_* \colon \IZ[K]_{\Psi}[[t]] \otimes_{\IZ[K]_{\Psi}[t,t^{-1}]} D[i]^{d_i+1 -*}
    \\
    \xrightarrow{\cong} \hom_{\IZ[K]_{\Psi}[[t^{-1}]]} (\IZ[K]_{\Psi}[[t^{-1}]]
    \otimes_{\IZ[K]_{\Psi}[t,t^{-1}]} D[i]_{d_i+1-*}, \IZ[K]_{\Psi}[[t^{-1}]]).
  \end{multline*}
  Since the homology groups of the finitely generated free $\IZ[K]_{\Psi}[[t^{-1}]]$-chain
  complex $\IZ[K]_{\Psi}[[t^{-1}]]\otimes_{\IZ[K]_{\Psi}[t,t^{-1}]} D[i]_*$ all vanish, it
  is contractible.  Therefore the $\IZ$-chain complex
  \[
  \hom_{\IZ[K]_{\Psi}[[t^{-1}]]} (\IZ[K]_{\Psi}[[t^{-1}]]
  \otimes_{\IZ[K]_{\Psi}[t,t^{-1}]} D[i]_{d_i+1_*}, \IZ[K]_{\Psi}[[t^{-1}]])
  \]
   is  contractible.  Hence all the homology groups of
  $\IZ[K]_{\Psi}[[t]] \otimes_{\IZ[K]_{\Psi}[t,t^{-1}]} D[i]^{d_i+1 -*}$ vanish. This
  finishes the proof of Lemma~\ref{lem:FP_and_duals}.
\end{proof}

    \begin{remark}\label{rem:Novikov_homology_appearing_in_proof}
      Note that in the proof of Lemma~\ref{lem:FP_and_duals} we had to use the Novikov
      homology since we do not know how to relate $\res_{G_i}^{K_i} D[i]^{d_i+1 -*}$
      directly to $\res_{G_i}^{K_i} D[i]_*$.
    \end{remark}

	\subsection{Proof of Theorem~\ref{the:products_and_fibring}}
    Now we are ready to give the proof of Theorem~\ref{the:products_and_fibring}.
    \begin{proof}[Proof of of Theorem~\ref{the:products_and_fibring}]
      We begin with constructing the desired manifolds $M_1$ and $M_2$.  Recall the we
      have fixed two distinct primes $l_1$ and $l_2$. Consider $i \in \{1,2\}$.  Since
      $K_i$ is finitely presented by assumption, we can choose a model for $BK_i$ whose
      $2$-skeleton $(BK_i)_2$ is finite.  If $g_i \in G$ is an element which is sent under
      $\phi$ to $1$ and $f_i \colon BK_i \to BK_i$ is a cellular homotopy equivalence
      inducing on $\pi_1(BK_i) = K_i$ the homomorphism sending $k$ to $g_ikg_i^{-1}$, then
      the mapping torus $T_{f_i}$ is a model for $BG_i$. Let
      $f_i|_{(BK_i)_2} \colon (BK_i)_2 \to (BK_i)_2$ be the restriction of $f_i$ to the
      $2$-skeleton of $BG_i$. Put
      \[
        Y_i = T_{f_i|_{(BK_i)_2}}.
      \]
      Since the inclusion $(BK_i)_2 \to BK_i$ is $2$-connected, the inclusion
      $Y_i \to T_{f_i} = BG_i$ is $2$-connected and we get an identification
      $\pi_1(Y_i) = G_i$. Let $q_i \colon \overline{Y_i} \to Y_i$ be the infinite cyclic
      covering associated to $\phi_ i \colon \pi_1(Y_i) = G_i \to \IZ$. Then
      $\overline{Y_i}$ is homotopy equivalent to $(BK_i)_2$ and hence homotopy equivalent
      to a finite $2$-dimensional $CW$-complex.

      Consider $i \in \{1,2\}$. Let $Z_i$ be the mapping cone of a cellular map
      $S^2 \to S^2$ of degree $l_i$, and let $z_i$ be a basepoint for $Z_i$.  Obviously
      $Z_i$ is a simply connected finite $3$-dimensional $CW$-complex satisfying
      \begin{eqnarray}
        H_n(Z_i)
        & = &
              \begin{cases}
                \IZ & \text{if}\; n = 0;
                \\
                \IZ/l_i & \text{if}\; n = 2;
                \\
                \{0\} & \textup{otherwise}.
              \end{cases}
                        \label{H_n(Z_i)}
      \end{eqnarray}

      Now define a connected finite $3$-dimensional $CW$-complex $X_i$
      \[
        X_i = Y_i \vee Z_i.
      \]
      Note that the inclusion $Y_i \to X_i$ is $2$-connected.
     
      Since $d_i \ge 6$, we can choose an embedding $X_i \to \IR^{d_i+1}$ and consider a
      regular neighborhood $N_i$ of $X_i $ in $\IR^{d_i+1}$.  This is a compact
      $(d_i+1)$-dimensional manifold such that the inclusion $X_i \to N_i$ is a homotopy
      equivalence and the inclusion $\partial N_i \to N_i$ is $(d_i -3)$-connected.  Put
      \[
        M_i = \partial N_i.
      \]
      In particular we get identifications
      $G_i = \pi_1(Y_i) =\pi_1(X_i) = \pi_1(N_i) = \pi_1(M_i)$.
     
      Next we prove the desired assertions.
      \\[1mm]~\ref{the:products_and_fibring:M_and_N:dimension_and_fundamental_group} This
      follows directly from the construction.
      \\[1mm]~\ref{the:products_and_fibring:not_F_2} Since the inclusions $M_i \to N_i$
      and $X_i \to N_i$ are $3$-connected, it suffices to show that $X_i$ does not
      virtually $\F_2$-fibre. Consider any finite covering
      $f_i \colon \widehat{X_i} \to X_i$.  Let $\widehat{Y_i}$ be the preimage of $Y_i$
      under $f_i$.  Next consider any group epimorphism
      $\psi_i \colon G_i = \pi_1(\widehat{X_i}) \to \IZ$.  Let
      $p_i \colon \overline{X_i} \to \widehat{X_i}$ be the infinite cyclic covering
      associated to $\psi_i$.  Denote by $q_i \colon \overline{Y_i} \to \widehat{Y_i}$ its
      restriction to $\widehat{Y_i}$.  Then
      $H_2(\overline{X_i}) \cong_{\IZ} H_2(\overline{Y_i}) \oplus
      H_2(\overline{X_i},\overline{Y_i})$ and $H_2(\overline{X_i},\overline{Y_i})$ is
      isomorphic to $\bigoplus_I \IZ/l_i$ for some infinite index set $I$. Hence
      $H_2(\overline{X_i})$ is not finitely generated as an abelian group.  Therefore
      $\overline{X_i}$ is not of type $\F_2$. This shows that $X_i$ and hence $M_i$ does
      not $\F_2$-fibre.  \\[1mm]~\ref{the:products_and_fibring:finitely_fibres} This is
      the hard part of the proof.
     
      We want to show that the infinite cyclic covering
      $p \colon \overline{M_1 \times M_2} \to M_1 \times M_2$ associated to the
      epimorphism $\Phi$ defined in~\eqref{choice:of_Phi} has the property that the total
      space $\overline{M_1 \times M_2}$ is homotopy equivalent to a finitely dominated
      $CW$-complex.

      Recall that $K$ is the kernel of $\phi$.  Since there exists an exact sequence of
      groups $1 \to K_1 \times K_2 \to K \to \IZ \to 1$ and $K_1$ and $K_2$ are finitely
      presented by assumption, $K$ is finitely presented.  Obviously
      $K \cong \pi_1(\overline{M_1 \times M_2})$.

      We conclude from~\cite{Wall(1966)} or~\cite[Proposition~11.11 on page~222 and
      Proposition~14.9 on page~282]{Lueck(1989)} that $\overline{M_1 \times M_2}$ is
      homotopy equivalent to a finitely dominated $CW$-complex if and only if the cellular
      $\IZ[K]$-chain complex of the universal covering of $\overline{M_2 \times M_2}$,
      which is the restriction $\res_{G_1 \times G_2}^K C_*(\widetilde{M_1 \times M_2})$
      from $G_1 \times G_2$ to $K$ of the cellular $\IZ[G_1 \times G_2]$-chain complex
      $C_*(\widetilde{M_1 \times M_2})$, is homotopically of type $\FP$. So it remains to
      show that the $\IZ[K]$-chain complex
      $\res_{G_1 \times G_2}^K C_*(\widetilde{M_1}) \otimes_{\IZ} C_*(\widetilde{M_2})$ is
      homotopically of type $\FP$.

      We get for $i = 1,2$ a $\IZ[G_i]$-chain isomorphism
      \[
        C_*(\widetilde{X_i}) \cong D[i]_* \oplus \IZ[G_i] \otimes_{\IZ} E[i]_*
      \]
      if we put $D[i]_* = C_*(\widetilde{Y_i})$ and $E[i]_* = C_*(Z_i,\{z_i\})$. Since
      $K_i = \pi_1(\overline{Y_i})$, the cellular $\IZ[\pi_1(\overline{Y_i})]$-chain
      complex of the universal covering of $\overline{Y_i}$ is
      $\res_{G_i}^K C_*(\widetilde{Y_i}) $ and $\overline{Y_i}$ is homotopy equivalent to
      a finite $CW$-complex, $\res_{G_i}^K D[i]_*$ is homotopically of type $\FP$.  We
      conclude for the dual $D[i]^{d_i +1- *}$ of $D[i]_*$ from
      Lemma~\ref{lem:FP_and_duals} that $\res_{G_1 \times G_2}^{K_i} D[i]^{d_i +1- *}$ is
      homotopically of type $\FP$. The dual $(\IZ[G] \otimes E[i]_*)^{d_i + 1 -*}$ of
      $\IZ[G] \otimes E[i]_*$ is of the form $\IZ[G_i] \otimes_{\IZ} E[i]^{d_i +1 -*}$ for
      the finite free $\IZ$-chain complex $E[i]^{d_i +1 -*}$ given by the dual of
      $E[i]_*$.  We conclude that the dual $C^{d_i + 1-*}(\widetilde{X_i})$ of
      $C_*(\widetilde{X})$ is $\IZ[G_i]$-isomorphic to
      $D[i]^{d_i +1- *} \oplus (\IZ[G] \otimes E[i]^{d_i + 1 -*})$ and hence is special.
      Since $C_*(\widetilde{N_i})$ is $\IZ[G_i]$-chain homotopy equivalent to
      $C_*(\widetilde{X_i})$, the dual $C^{d_i +1 -*}(\widetilde{N_i})$ of
      $C_*(\widetilde{N_i})$ is a special finite free $\IZ[G_i]$-chain complex.  By
      Poincar\'e duality there is a $\IZ[G_i]$-chain homotopy equivalence
      $C^{d_i + 1 - *}(\widetilde{N_i}) \xrightarrow{\simeq}
      C_*(\widetilde{N_i},\widetilde{M_i})$.  So the upshot of this discussion is that
      both finite free $\IZ[G_i]$-chain complexes $C_*(\widetilde{N_i})$ and
      $C_*(\widetilde{N_i},\widetilde{M_i})$ are special.

      Note that for $i = 1,2$ we have the exact sequence of finite free $\IZ[G_i]$-chain
      complexes
      \[
        0 \to C_*(\widetilde{M_i}) \to C_*(\widetilde{N_i}) \to
        C_*(\widetilde{N_i},\widetilde{M_i}) \to 0.
      \]
      Now Lemma~\ref{lem:six_out_of_four} and
      Lemma~\ref{lem:products_of_special_complexes} imply that
      $\res^K_{G_1 \times G_2} C_*(\widetilde{M_1}) \otimes_{\IZ} C_*(\widetilde{M_2})$ is
      homotopically of type $\FP$.  This finishes the proof of
      Theorem~\ref{the:products_and_fibring}.
    \end{proof}

    \begin{question}\label{que:virtually_fibring_and_bundles_and_products}\
      \begin{enumerate}
      \item\label{lque:virtually_fibring_and_bundles_and_products:fibring} Are there
        examples of aspherical smooth closed manifolds $M_1$ and $M_2$ such that
        $M_1 \times M_2$ fibres over $S^1$ but neither $M_1$ nor $M_2$ fibre over $S^1$?
       
      \item\label{que:virtually_fibring_and_bundles_and_products:virtually_fibring} Are
        there examples of aspherical smooth closed manifolds $M_1$ and $M_2$ such that
        $M_1 \times M_2$ virtually fibres over $S^1$ but neither $M_1$ nor $M_2$ virtually
        fibres over $S^1$?

      \end{enumerate}
    \end{question}

\begin{remark}
  The examples of non virtually fibring manifolds in~\cite{Avramidi-Okun-Schreve(2024)} seem like a
  good place to look for counterexamples to Question~\ref{que:virtually_fibring_and_bundles_and_products}.
 \end{remark}


    \typeout{-------------------------- Appendix -------------------------}

\begin{appendix}
    \section{Survey on virtually fibring 3-manifolds}\label{sec:Leftovers}
We conclude with a short survey on virtually fibring $3$-manifolds.  We make no claim to originality here.

\begin{theorem}\label{the:virtually_fibring_and_3-manifolds}\

  \begin{enumerate}

  \item\label{the:virtually_fibring_and_3-manifolds:RFRS_not_graph} Let $M$ be an
    irreducible $3$-manifold with infinite fundamental group. Then $\pi_1(M)$ is a
    \RFRS-group if and only if one of the following conditions is satisfied:

    \begin{enumerate}
       
    \item $M$ is not a graph manifold;

    \item $M$ is a non-positively curved graph manifold;

    \end{enumerate}

  \item\label{the:virtually_fibring_and_3-manifolds:RFRS_imples_virt_fib} If $M$ is an
    irreducible $3$-manifold whose fundamental group is a non-trivial \RFRS-group, then
    $M$ virtually fibres over $S^1$;

  \item\label{the:virtually_fibring_and_3-manifolds:RFRS_non-positively_curved}
       
    Let $M$ be an aspherical closed $3$-manifold. Then $\pi_1(M)$ is \RFRS-group if and
    only if $M$ is non-positively curved;

  \item\label{the:virtually_fibring_and_3-manifolds:necessary} Let $M$ be a closed
    $3$-manifold which virtually fibres over $S^1$.

    Then either $M$ is aspherical or a finite covering of $M$ is homeomorphic to
    $S^1 \times S^2$.

  \end{enumerate}
\end{theorem}
\begin{proof}~\ref{the:virtually_fibring_and_3-manifolds:RFRS_not_graph} The fundamental
  group $\pi_1(M)$ is a \RFRS-group by Agol~\cite{Agol(2013)} and Przytycki and
  Wise~\cite{Przytycki-Wise(2018)} and Wise~\cite{Wise(2012hierachy)} if $M$ is not a
  graph manifold. If $M$ is a graph manifold then $\pi_1(M)$ is $\RFRS$-group if and only
  if $M$ is non-positively curved, see Liu~\cite{Liu(2013)}.
  \\[1mm]~\ref{the:virtually_fibring_and_3-manifolds:RFRS_imples_virt_fib} This is proved
  by Agol~\cite{Agol(2008)}.
  \\[1mm]~\ref{the:virtually_fibring_and_3-manifolds:RFRS_non-positively_curved} By the
  Sphere Theorem\cite[Theorem 4.3]{Hempel(1976)}, an irreducible $3$-manifold is
  aspherical if and only if has infinite fundamental group.  Using the prime decomposition
  and the fact that every prime manifold which is not irreducible is finitely covered by
  $S^1 \times S^2$ one concludes that a closed $3$-manifold is aspherical if and only if
  it is an irreducible $3$-manifold with infinite fundamental group.
   
  An irreducible $3$-manifold, which is not a graph manifold and has infinite fundamental
  group, is non-negatively curved, see Leeb~\cite{Leeb(1995)}.  Now the claim follows from
  assertion~\ref{the:virtually_fibring_and_3-manifolds:RFRS_imples_virt_fib}.
  \\[1mm]~\ref{the:virtually_fibring_and_3-manifolds:necessary} If $M$ virtually fibres
  over $S^1$, all its $L^2$-Betti numbers must vanish by L\"uck~\cite{Lueck(2002)}.  Now
  apply Lott--L\"uck~\cite{Lott-Lueck(1995)}.
\end{proof}

So the only class of closed $3$-manifolds, where we cannot decide in general whether they
virtually fibre over $S^1$ is the case of a graph manifold whose fundamental group is not
{\RFRS}.

Note that there are graph manifolds whose fundamental group is not {\RFRS}, whose geometry
is \Sol, and which virtually fibre over $S^1$, see Agol~\cite{Agol(2013)}.

Let $M$ be a closed Seifert manifold. Then its geometry is $S^3$, $\IR^3$,
$S^2 \times \IR$, $\IH^2 \times \IR$, $\Nil$, or $\widetilde{\SLpar{2}{\IR}}$.  There is a
finite covering $\overline{M} \to M$ and an $S^1$-principal bundle
$p \colon \overline{M} \to S $ for a closed orientable surface $S$ such that we get for
the Euler class $e = e(p)$ and the Euler characteristic $\chi = \chi(S)$
\[\begin{array}{c|ccc} &\chi > 0 & \chi =0 &
                                             \chi < 0 \\ \hline
    e =0  & S^2 \times \IR & \IR^3 & \IH^2 \times \IR \\
    e \ne 0  & S^3 & \Nil & \widetilde{\SLpar{2}{\IR}}
  \end{array}
\]
Hence $M$ does not virtually over $S^1$ if and only if the geometry is $S^3$ or
$\widetilde{SL_2(\IR)}$ by the following argument. If the geometry if $S^3$, the
fundamental group is finite and $M$ does not virtually fibre over $S^1$. If the geometry
is $\widetilde{SL_2(\IR)}$, Example~\ref{exa:B-a_surface} shows that $M$ does not
virtually fibre over $S^1$. In all other cases $\overline{M}$ fibres over $S^1$ by
Lemma~\ref{lem:virtually_fibring_and_bundles_and_products_inheritiance_from_base_total-space}~%
\ref{lem:virtually_fibring_and_bundles_and_products_inheritiance_from_base_total-space:virtually_fibring}.
   
\end{appendix}



\addcontentsline{toc<<}{section}{References} \bibliographystyle{abbrv}


\end{document}